\documentclass{siamonline1116}

\usepackage{lipsum}
\usepackage{amsfonts}
\usepackage[mathscr]{euscript}
\usepackage{graphicx}
\usepackage{epstopdf}
\usepackage{algorithm}
\usepackage{algorithmic}
\usepackage{amsmath}
\usepackage{amssymb}
\usepackage{latexsym}
\usepackage{bbm}
\usepackage[caption=false]{subfig}
\ifpdf
  \DeclareGraphicsExtensions{.eps,.pdf,.png,.jpg}
\else
  \DeclareGraphicsExtensions{.eps}
\fi

\numberwithin{theorem}{section}

\newcommand{\TheTitle}{A Discrete Empirical Interpolation Method for Interpretable Immersion and Embedding of Nonlinear Manifolds}
\newcommand{\TheAuthors}{S. E. Otto and C. W. Rowley}

\headers{\TheTitle}{\TheAuthors}

\title{{\TheTitle}\thanks{Last updated \today.
\funding{This work was funded by ARO award W911NF-17-0512 and DARPA.}}}

\author{
  Samuel E. Otto\thanks{Mechanical and Aerospace Engineering, Princeton University, Princeton, NJ
    (\email{sotto@princeton.com}).}
  \and
  Clarence W. Rowley\thanks{Mechanical and Aerospace Engineering, Princeton University, Princeton, NJ (\email{cwrowley@princeton.edu}).}
}

\usepackage{amsopn}

\DeclareMathOperator*{\argmax}{\arg\!\max\enskip}

\DeclareMathOperator*{\argsortDown}{argsort\downarrow\enskip}

\DeclareMathOperator{\vspan}{span}
\DeclareMathOperator{\Range}{Range}

\DeclareMathOperator*{\Expectation}{\!\mathbb{E}}
\DeclareMathOperator{\Vol}{Vol}
\DeclareMathOperator{\house}{house}
\DeclareMathOperator{\length}{length}

\newsiamremark{example}{Example}
\newsiamremark{remark}{Remark}

\makeatletter
\renewcommand*\env@matrix[1][*\c@MaxMatrixCols c]{%
  \hskip -\arraycolsep
  \let\@ifnextchar\new@ifnextchar
  \array{#1}}
\makeatother

\ifpdf
\hypersetup{
  pdftitle={\TheTitle},
  pdfauthor={\TheAuthors}
}
\fi

\usepackage{xr}

\begin{document}

\maketitle

\begin{abstract}
  Manifold learning techniques seek to discover structure-preserving mappings of high-dimensional data into low-dimensional spaces. 
  While the new sets of coordinates specified by these mappings can closely parameterize the data, they are generally complicated nonlinear functions of the original variables. This makes them difficult to interpret physically. 
  Furthermore, in data-driven model reduction applications the governing equations may have structure that is destroyed by nonlinear mapping into coordinates on an inertial manifold, creating a computational bottleneck for simulations.
  Instead, we propose to identify a small collection of the original variables which are capable of uniquely determining all others either locally via immersion or globally via embedding of the underlying manifold. 
  When the data lies on a low-dimensional subspace the existing discrete empirical interpolation method (DEIM) accomplishes this with recent variants employing greedy algorithms based on pivoted QR (PQR) factorizations.
  However, low-dimensional manifolds coming from a variety of applications, particularly from advection-dominated PDEs, do not lie in or near any low-dimensional subspace.
  Our proposed approach extends DEIM to data lying near nonlinear manifolds by applying a similar pivoted QR procedure simultaneously on collections of patches making up locally linear approximations of the manifold, resulting in a novel simultaneously pivoted QR (SimPQR) algorithm. 
  An optimization problem is solved at each stage to greedily select a coordinate and the manifold patches over which the coordinate is used for pivoting. 
  The immersion provided by applying SimPQR to the tangent spaces can be extended to an embedding by applying SimPQR a second time to a modified collection of vectors.
  The SimPQR method for computing these ``nonlinear DEIM'' (NLDEIM) coordinates for immersion and embedding is successfully applied to real-world data lying near an inertial manifold in a cylinder wake flow as well as data coming from a viscous Burgers equation with different initial conditions.
\end{abstract}



\section{Introduction}
In this paper, we consider the problem: \textit{given a low-dimensional surface in a high-dimensional euclidean space, how can we identify a small subset of the original coordinates on which the others always depend?}
Originally, we asked this question in the context of creating low-dimensional ``reduced order'' models of high-dimensional dynamical systems arrising from discretized simulations of fluid flows.
However, we also think that the work presented herein might be useful as a tool for unsupervised selection of easily interpretable features for data-science applications.

Accurately modeling fluid flows and other phenomena governed by  spatio-temporal partial differential equations requires a large computational effort to solve the discretized equations on a sufficiently fine grid.
Data-driven reduced order modeling entails using data from previous experiments or simulations to build simplified models with fewer degrees of freedom that still capture essential aspects of the full problem.
This is usually accomplished by identifying an inertial manifold that captures the long-time behavior of the system and restricting the modeled dynamics to evolve on its surface \cite{Rega2005dimension}.

Predominant approaches in solid and fluid dynamics use data or the governing equations to construct one or more low-dimensional subspaces that are used to approximate the dynamics by Galerkin or Petrov-Galerkin projection. 
A survey of such methods is presented in \cite{Benner2015survey}.
One data-driven approach frequently used for model reduction in fluid and solid dynamics is to construct an optimal subspace capturing the solution's variance using Principal Component Analysis (PCA) \cite{Hotelling1933analysis} which is also known as Proper Orthogonal Decomposition (POD) or the Karhunen-Lo{\`e}ve transformation \cite{Karhunen1947lineare, Loeve1978probability}.
However, naive projection of the nonlinear terms in the governing equations still requires evaluation of these terms at all points in the spatial domain followed by inner products with the chosen basis functions.
This computational bottleneck is overcome by using interpolation methods \cite{Everson1995karhunen,Barrault2004empirical,Nguyen2008best,Chaturantabut2010} to reconstruct the state or the nonlinear terms using a small collection of carefully selected state variables often corresponding to points in the spatial domain.
The Discrete Empirical Interpolation Method (DEIM) \cite{Chaturantabut2010, Drmac2016, Drmac2017} is one-such technique that has enabled development of a wide variety of reduced-order modeling techniques over the last decade.
DEIM employs greedy algorithms based on pivoted factorizations of the leading POD modes to efficiently choose a nearly optimal collection of interpolation points.

The problem is that even very simple data sets coming from fluid dynamics and other advection-dominated problems cannot be captured in a low-dimensional subspace \cite{Ohlberger2016reduced}.
For example, a moving localized disturbance cannot be parsimoniously represented as a superposition of modes, even though such solutions lie on a one-dimensional manifold and can often be distinguished reliably by measurements at a small number of spatial grid points.
In general, the solutions to the governing equations might lie on or near a low-dimensional manifold that is curved in such a way that low-dimensional subspaces do a poor job approximating it.
One category of approaches employed in \cite{Dihlmann2011model, Amsallem2012nonlinear, Peherstorfer2014localized, Peng2016nonlinear} utilizes local POD bases in state and/or parameter space that can be combined with their own individual DEIM-based reduced order models.
However, these piecewise linear methods require switching, multiple sets of interpolation points, and do not smoothly capture the manifold structure.
In order to remedy these problems, model reduction techniques that project the state dynamics onto nonlinear manifolds learned from data \cite{Lee2018model, Pyta2016nonlinear, Winstead2017nonlinear} are being developed.
At present, they suffer from essentially the same problems as POD-Galerkin methods before the advent of DEIM, namely, computations must still be performed over the entire spatial domain.

In this paper, we systematically develop an efficient ``Nonlinear DEIM'' (NLDEIM) algorithm that identifies a small collection of state variables (measurements at specific spatial locations in the case of discretized PDEs) that are sufficient to immerse a given state space manifold.
This means that any of the manifold's tangent spaces has a 1-1 orthogonal projection into the chosen coordinates, allowing for reconstruction of state time derivatives tangent to the manifold using only these coordinates.
Our approach, which we call ``Simultaneously Pivoted QR'' (SimPQR), performs a pivoted Householder QR factorization on a collection of orthonormal tangent space bases where the pivot coordinates are shared simultaneously across multiple tangent spaces.
The technique includes a parameter that allows the user to adjust the trade-off between greedily choosing the smallest set of coordinates that immerse the manifold at one extreme and choosing the collection of DEIM coordinates over every tangent space at the other extreme.
By immersing the inertial manifolds identified using data-driven manifold learning in a small set of the original state variables, we hope to enable DEIM-like model reduction to be applied in this setting, however this is beyond the scope of the present work.

While immersing the underlying manifold captures its tangent spaces in a 1-1 fashion, there may be a discrete set of points on the manifold that project to the same point in the reduced set of coordinates, hence we don't necessarily have an embedding yet.
By applying the same SimPQR algorithm to the collection of vectors separating points on these discrete branches of the manifold, we show than an additional set of coordinates is found that, taken together with the original set of immersing coordinates, embeds the manifold.
We think that embedding nonlinear manifolds via projection into a small set of the original coordinates could provide superior interpretability compared with other linear and nonlinear dimensionality reduction techniques that introduce transformations of the original features in data-science applications.
The algorithms discussed herein may be thought of as performing a type of unsuperised feature selection for high-dimensional data lying on nonlinear manifolds.

There are a wide variety of techniques for nonlinear dimensionality reduction that could be used in conjunction with the proposed NLDEIM framework to identify optimal immersion and embedding coordinates for model reduction and feature selection applications.
Some examples include kernel principal component analysis (KPCA) \cite{Scholkopf1998}, Isomap \cite{Tenenbaum2000global}, locally linear embedding (LLE) \cite{Roweis2000nonlinear}, Laplacian eigenmaps \cite{Belkin2003}, diffusion maps \cite{Coifman2006}, and autoencoders \cite{Hinton2006reducing}.
In order to be used in conjunction with NLDEIM, the manifold learning algorithm must be able to construct tangent spaces at a collection of sampled data points.
Since we will use the Laplacian eigemaps algorithm on many of our examples due to its simplicity, we provide an overview of this algorithm along with a technique for interpolating its coordinates and finding tangent spaces in \cref{app:LapEigAndTanSpace}.
We are not aware of other authors using kernel-based manifold learning techniques like KPCA, Laplacian eigenmaps, or diffusion maps to construct tangent spaces, however it should be possible to extend our approach to these other methods too.
The main idea is to construct a smooth function that locally interpolates the embedding coordinates near training data points and differentiate it to find a basis for the tangent space.

A variety of related methods have been developed in \cite{Broomhead2005dimensionality, Jamshidi2007towards, Hegde2015numax, Sun2017sparse, Kvinge2018too, Kvinge2018gpu} that utilize the collection of secant vectors within a collection of data points in order to identify an optimal linear embedding of these data.
However, it is a well-known problem that the methods do not scale well for large data sets since the number of secant vectors grows with the square of the number of data points.
In contrast, our proposed nonlinear DEIM framework utilizes the tangent spaces to the manifold and scales as $\mathcal{O}(K\log K)$ where $K$ is the number of sampled tangent space usually taken to be proportional or equal to the number of training data points.
The secant-based methods mentioned above also differ from our proposed nonlinear DEIM technique in that they identify an optimal projection subspace for the embedding rather than a subset of the original coordinates.
Therefore, they will not yield a substantial improvement in the efficiency of model reduction techniques based on manifold Galerkin projection of the governing equations, as proposed in \cite{Lee2018model}, whereas nonlinear DEIM has the potential to do so.

This paper is organized as follows:
In \cref{sec:DEIM} we review some fundamental results from the literature and the commonly used Q-DEIM algorithm based on pivoted QR (PQR) factorization presented in \cite{Drmac2016}.
This discussion lays the groundwork for our development of a novel simultaneously pivoted QR (SimPQR) algorithm for identifying coordinate immersions in \cref{sec:immersion_coordinates}.
In \cref{subsec:simPQRintro}, the SimPQR method is motivated and systematically developed.
The SimPQR algorithm requires solving an optimization problem for the pivoting during each stage whose solution method is developed in \cref{subsec:optimalPivoting}.
The solution path for our SimPQR algorithm over the range of user-defined parameter values is explored in \cref{subsec:GammaPath} where we find that the next largest and smallest values of the parameter producing a change in the solution can be computed with little additional cost.
Finally in \cref{subsec:theoreticalGuarantees} we present some guidelines for choosing the user-defined parameter in the form of theoretical guarantees regarding the trade-off between number of coordinates and tangent space reconstruction robustness.
The novel approach developed in \cref{sec:BranchingCoordinates} shows that SimPQR can be applied a second time to modified collections of vectors in order to identify additional coordinates that extend the previously computed immersion to an embedding.
In \cref{sec:applications} we apply our methods to identify spatial sampling locations embedding the flow behind a cylinder in \cref{subsec:CylinderWakeSampling} and solutions of the viscous Burgers equation starting from different initial conditions in \cref{subsec:Burgers}.
The Nonlinear DEIM methods using SimPQR are compared against DEIM methods using PQR.
Additional toy examples used to illustrate different aspects of Nonlinear DEIM throughout the paper are presented in \cref{app:ToyModels}.

\section{Linear manifolds and the Discrete Empirical Interpolation Method}
\label{sec:DEIM}
If the manifold under consideration is a subspace, then a solution strategy called the Discrete Empirical Interpolation Method or DEIM \cite{Chaturantabut2010} is already well-understood \cite{Drmac2017}.
Suppose that the data lies near ar $r$-dimensional subspace spanned by the columns of the orthonormal matrix $\mathbf{U}\in\mathbb{R}^{n\times r}$.
To measure $d$-components, $\mathscr{P} = \lbrace j_1,\ldots,j_d \rbrace \subseteq \lbrace 1,\ldots, n \rbrace$, of our data vectors $\mathbf{x} = \mathbf{U}\mathbf{z} + \mathbf{n}\in\mathbb{R}^n$, let us introduce the sampling matrix $\boldsymbol{\Pi}_{\mathscr{P}} = \begin{bmatrix} \mathbf{e}_{j_1} & \cdots & \mathbf{e}_{j_d} \end{bmatrix}$ where $\mathbf{e}_{j_1},$ $\ldots,$ $\mathbf{e}_{j_d}$ are selected columns from the $n\times n$ identity matrix $\mathbf{I_n}$.
Thus, specific features corresponding to these indices are sampled according to
\begin{equation}
    \underbrace{\begin{bmatrix}
    x_{j_1} \\
    \vdots \\
    x_{j_d}
    \end{bmatrix}}_{\boldsymbol{\Pi}_{\mathscr{P}}^T\mathbf{x}} = 
    \underbrace{\begin{bmatrix}
    U_{j_1, 1} & \cdots & U_{j_1, r}\\
    \vdots & & \vdots \\
    U_{j_d, 1} & \cdots & U_{j_d, r}
    \end{bmatrix}}_{\boldsymbol{\Pi}_{\mathscr{P}}^T\mathbf{U}} 
    \underbrace{\begin{bmatrix}
    z_1 \\
    \vdots \\
    z_r
    \end{bmatrix}}_{\mathbf{z}} +
    \underbrace{\begin{bmatrix}
    n_{j_1} \\
    \vdots \\
    n_{j_d}
    \end{bmatrix}}_{\boldsymbol{\Pi}_{\mathscr{P}}^T\mathbf{n}}.
\end{equation}
Let $(\cdot)^{+}$ denote the Moore-Penrose pseudoinverse. If the sampling indices are chosen so that $\boldsymbol{\Pi}_{\mathscr{P}}^T\mathbf{U}$ is left-invertible, then we may approximately recover 
\begin{equation}
    \left(\boldsymbol{\Pi}_{\mathscr{P}}^T\mathbf{U}\right)^{+}\boldsymbol{\Pi}_{\mathscr{P}}^T\mathbf{x} = \mathbf{z} + \left(\boldsymbol{\Pi}_{\mathscr{P}}^T\mathbf{U}\right)^{+}\boldsymbol{\Pi}_{\mathscr{P}}^T\mathbf{n} \quad \Rightarrow \quad \mathbf{\hat{x}} = \mathbf{U}\left(\boldsymbol{\Pi}_{\mathscr{P}}^T\mathbf{U}\right)^{+}\boldsymbol{\Pi}_{\mathscr{P}}^T\mathbf{x},
    \label{eqn:linear_DEIM_recon}
\end{equation}
where $\mathbf{\hat{x}}$ is our estimate for $\mathbf{x}$ based on the sampled coordinates.
If $\boldsymbol{\Pi}_{\mathscr{P}}^T\mathbf{U}$ is not left-invertible then we will call the subspace, $\Range (\mathbf{U})$, ``$\boldsymbol{\Pi}_{\mathscr{P}}$-singular''. 
In the above, we could replace $(\cdot)^{+}$ with any left-inverse that may be chosen to minimize error in the above reconstruction given the probability distribution of the noise. 
The Moore-Penrose pseudoinverse is used here since it minimizes the reconstruction error when the noise $\mathbf{n}$ has mean zero and isotropic co-variance.

We see that the error
\begin{equation}
    \Vert \mathbf{x} - \mathbf{\hat{x}} \Vert_2 
    = \Vert \mathbf{U}\left(\boldsymbol{\Pi}_{\mathscr{P}}^T\mathbf{U}\right)^{+}\boldsymbol{\Pi}_{\mathscr{P}}^T\mathbf{n} \Vert_2
    \leq \Vert \left(\boldsymbol{\Pi}_{\mathscr{P}}^T\mathbf{U}\right)^{+} \Vert_2 \Vert \boldsymbol{\Pi}_{\mathscr{P}}^T\mathbf{n} \Vert_2
     \label{eqn:linear_DEIM_recon_error}
\end{equation}
depends on the amplification of the noise in the selected features due to $\left(\boldsymbol{\Pi}_{\mathscr{P}}^T\mathbf{U}\right)^{+}$.
Thus, it is desirable to choose the sampled features to maximize the smallest nonzero singular value $\sigma_{min}\left(\boldsymbol{\Pi}_{\mathscr{P}}^T\mathbf{U}\right)$.
As a surrogate for the smallest nonzero singular value, it is easier to try to make all of singular values large by attempting to maximize a matrix ``volume'',
\begin{equation}
    \Vol\left(\boldsymbol{\Pi}_{\mathscr{P}}^T\mathbf{U}\right) = \max_{\mathscr{I}\subseteq\lbrace 1,\ldots,d\rbrace,\ \vert \mathscr{I} \vert = r} \left\vert \det{\left( [\boldsymbol{\Pi}_{\mathscr{P}}^T\mathbf{U}]_{\mathscr{I},:} \right)} \right\vert.
\end{equation}
The matrix volume, $\Vol (\mathbf{M})$ where $\mathbf{M}\in\mathbb{R}^{p\times q}$, is defined to be the largest absolute determinant among the $\min\lbrace p,q \rbrace$ square sub-matrices of $\mathbf{M}$.
\Cref{prop:volumeSurrogate} shows that the volume defined above lower bounds the product of the singular values, which in turn lower bounds the smallest singular value of $\boldsymbol{\Pi}_{\mathscr{P}}^T\mathbf{U}$.
Therefore, it is a good surrogate when our goal is to make $\sigma_{min}\left(\boldsymbol{\Pi}_{\mathscr{P}}^T\mathbf{U}\right)$ large.
\begin{proposition}[Volume as a Surrogate]
\label{prop:volumeSurrogate}
Let $\sigma_1(\boldsymbol{\Pi}_{\mathscr{P}}^T\mathbf{U}) \geq \sigma_2(\boldsymbol{\Pi}_{\mathscr{P}}^T\mathbf{U}) \geq \cdots \geq \sigma_r(\boldsymbol{\Pi}_{\mathscr{P}}^T\mathbf{U}) \geq 0$ denote the largest singular values of the matrix $\boldsymbol{\Pi}_{\mathscr{P}}^T\mathbf{U}$ in descending order. 
Then,
\begin{equation}
    \sigma_{min}\left(\boldsymbol{\Pi}_{\mathscr{P}}^T\mathbf{U}\right) \geq 
    \prod_{l=1}^r \sigma_{l}\left(\boldsymbol{\Pi}_{\mathscr{P}}^T\mathbf{U}\right) \geq
    \Vol\left(\boldsymbol{\Pi}_{\mathscr{P}}^T\mathbf{U}\right).
\end{equation}
\end{proposition}
\begin{proof}
See \cref{app:Proofs}.
\end{proof}

The Q-DEIM algorithm \cite{Drmac2016} chooses exactly $d=r$ coordinates through a Householder QR factorization
\begin{equation}
    \mathbf{U}^T\begin{bmatrix}[c|c]
            \boldsymbol{\Pi}_{\mathscr{P}} & \boldsymbol{\Pi}_{\mathscr{P}^c} \end{bmatrix} = \mathbf{Q} \begin{bmatrix}[c|c] \mathbf{R}^{(1)} & \mathbf{R}^{(2)} \end{bmatrix},
    \label{eqn:PQR}
\end{equation}
employing Businger-Golub \cite{Businger1965linear} column pivoting.
This algorithm is given as Algorithm 5.4.1 in \cite{Golub2013matrix} and is reproduced here as \cref{alg:PQR}; will refer to it as ``Pivoted QR'' or PQR for short.
This pivoting strategy chooses the column 
\begin{equation}
    j^* = \argmax_{j\in\mathscr{P}^c}\mathbf{c}(j)
\end{equation}
at the $i$th stage that gives the largest diagonal entry $\vert \mathbf{R}^{(1)}(i,i) \vert = \sqrt{\mathbf{c}(j^*)}$. 
The entries $\mathbf{c}(j)$ keep track of the residual column magnitudes which become the candidate diagonal entries at each stage.
Therefore, \cite{Drmac2016} concludes that it can be thought of as a greedy algorithm for maximizing the volume
\begin{equation}
    \Vol\left(\boldsymbol{\Pi}_{\mathscr{P}}^T\mathbf{U}\right) = \Vol\left(\mathbf{R}^{(1)}\right) = \prod_{i=1}^r \vert \mathbf{R}^{(1)}(i,i) \vert.
\end{equation}
The application of Householder transformations to upper-triangularize $\mathbf{U}^T$ and the reconstruction of $\mathbf{Q}$, $\mathbf{R}^{(1)}$ and $\mathbf{R}^{(2)}$ from the output of \cref{alg:PQR} are discussed in \cref{app:Householder}.

As was shown in \cite{Drmac2016}, performing Q-DEIM using PQR yields the bound,
\begin{equation}
    \Vert \left(\boldsymbol{\Pi}_{\mathscr{P}}^T\mathbf{U}\right)^{-1} \Vert_2 \leq \sqrt{n-r+1} \frac{\sqrt{4^r+6r-1}}{3},
\end{equation}
Though they find that the second term only achieves its $\mathcal{O}(2^r)$ dependence on pathological matrices.
In practice the second term has $\mathcal{O}(r)$ dependence yielding the practical error bound
\begin{equation}
    \Vert \left(\boldsymbol{\Pi}_{\mathscr{P}}^T\mathbf{U}\right)^{-1} \Vert_2 \leq \mathcal{O} (r\sqrt{n-r}).
\end{equation}
The bound may be improved \cite{Drmac2017} to $\sqrt{1+\eta^2r(n-r)}$ for any $\eta\geq 1$ through a Strong Rank-Revealing QR factorization \cite{Gu1996} that employs additional column pivoting operations based on maximizing the volume.
However, more pivoting operations are required for smaller $\eta$.

\begin{algorithm}
  \caption{Householder QR with Column Pivoting \cite{Golub2013matrix}}
  \label{alg:PQR}
  \begin{algorithmic}
    \STATE{Given: $\mathbf{A} = \mathbf{U}^T$}
    \COMMENT{initialize matrix (to be over-written)}
    \STATE{$\mathscr{P} = []$ and $\mathscr{P}^c = [1,\ldots,n]$}
    \COMMENT{initialize pivot index lists}
    \STATE{$\mathbf{c}(j) = \Vert \mathbf{A}(:,j) \Vert_2^2$ for $j=1,\ldots,n$}
    \COMMENT{compute initial residual norms}
    \STATE{$i = 1$}
    \COMMENT{QR progress}
    \WHILE[continue while there are nonzero residuals]{$\exists j$ s.t. $\sqrt{c(j)} > \epsilon$}
        \STATE{$j^* = \argmax_{j\in\mathscr{P}^c}{\mathbf{c}(j)}$}
        \COMMENT{pick pivot column}
        \STATE{$(\mathbf{v},\ \beta) = \house (\mathbf{A}(i:r,j^*))$}
        \COMMENT{Householder transform using \cref{alg:House}}
        \STATE{$\mathbf{A}(i:r, \mathscr{P}^c) \leftarrow (\mathbf{I}_{r-i+1} - \beta \mathbf{v}\mathbf{v}^T) \mathbf{A}(i:r, \mathscr{P}^c)$}
        \COMMENT{transform remaining columns}
        \STATE{$\mathbf{A}(i+1:r,j^*) \leftarrow \mathbf{v}(2:r-i+1)$}
        \COMMENT{store vector in zeroed entries of piv col}
        \FOR[loop over remaining columns]{$j\in\mathscr{P}^c$}
            \STATE{$\mathbf{c}(j) \leftarrow \mathbf{c}(j) - \mathbf{A}(i,j)^2$}
            \COMMENT{update residual magnitudes}
        \ENDFOR
        \STATE{$\mathscr{P} \leftarrow [\mathscr{P}(:), j^*]$ and $\mathscr{P}^c \leftarrow \mathscr{P}^c\setminus [j^*]$}
        \COMMENT{update pivot index sets}
        \STATE{$i \leftarrow i + 1$}
        \COMMENT{update QR progress}
    \ENDWHILE
    \RETURN{$\mathscr{P}$, $\mathscr{P}^c$, and $\mathbf{A}$}
  \end{algorithmic}
\end{algorithm}

\section{Discovering Immersion Coordinates for Nonlinear Manifolds}
\label{sec:immersion_coordinates}
Extending DEIM to nonlinear manifolds requires us to find a collection of the original features $x_{j_1},$ $\ldots,$ $x_{j_d}$ that determine the point $\mathbf{x}\in\mathcal{M}$ uniquely. 
\Cref{fig:2dSurfIn10D_TangentPlanes} shows two 3D projections of a 2D manifold in $\mathbb{R}^{10}$ in which each coordinate is a function of the first and last coordinates $x_1$ and $x_{10}$.
This entails two criteria, the first is local in character and the second is global.
We shall take up the local condition in this section and arrive at a set of coordinates that are sufficient to paramterize the manifold locally, thereby producing an immersion.
The global condition will be taken up in \cref{sec:BranchingCoordinates} and will be used to settle discrete ambiguities arising from the global manifold geometry, extending the immersion to an embedding.
\begin{figure}
    \centering
    \includegraphics[width=0.48\linewidth]{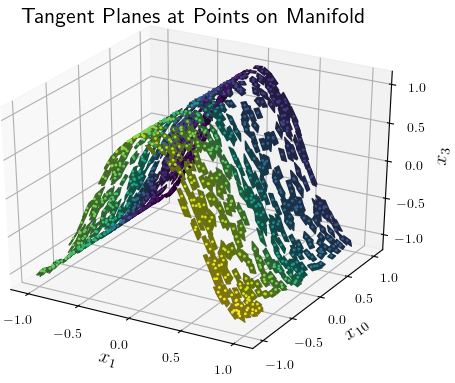}
    \hfill
    \includegraphics[width=0.48\linewidth]{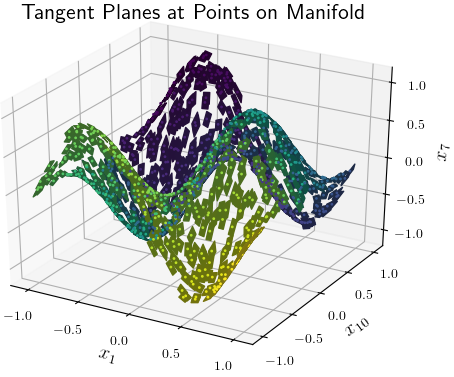}
    \caption{3-dimensional projections of synthetic data lying on a 2-dimensional manifold \cref{eqn:2DsurfMfd} embedded in $\mathbb{R}^{10}$ described fully by the first and last coordinates $x_1$ and $x_{10}$. The data is colored according to the leading graph Laplacian eigenfunction. Tangent planes at each data point are shown.}
    \label{fig:2dSurfIn10D_TangentPlanes}
\end{figure}

For simplicity, let us assume that $\mathcal{M}$ is a smooth, compact, manifold embedded in $\mathbb{R}^n$. 
Locally, we may approximate the manifold near $\mathbf{\bar{x}}\in\mathcal{M}$ using its tangent plane $\mathbf{\bar{x}} + T_{\mathbf{\bar{x}}}\mathcal{M}$.
Such tangent planes on two different example manifolds projected into three dimensions are shown in \cref{fig:2dSurfIn10D_TangentPlanes} and \cref{fig:SpiralMfd_ImmersionDiagram_Planes}.
In order to produce the desired immersion, our chosen set of coordinate indices $\mathscr{P}$ must act like DEIM coordinates in the sense of \cref{sec:DEIM} on each such tangent plane.
In other words, if $T_{\mathbf{\bar{x}}}\mathcal{M}$ is spanned by the orthonormal matrix $\mathbf{U}_{\mathbf{\bar{x}}}$, then $\boldsymbol{\Pi}_{\mathscr{P}}^T\mathbf{U}_{\mathbf{\bar{x}}}$ must be left-invertible $\forall \mathbf{\bar{x}}\in\mathcal{M}$.
Therefore, it is possible to find a neighborhood $\mathcal{U}_{\mathbf{\bar{x}}}\subseteq\mathcal{M}$ around any point $\mathbf{\bar{x}}\in\mathcal{M}$ so that the immersion $\boldsymbol{\Pi}_{\mathscr{P}}^T$ is a 1-1 function on $\mathcal{U}_{\mathbf{\bar{x}}}$.
Observe that in the spiral shaped manifold in \cref{fig:SpiralMfd_ImmersionDiagram_Planes}, this can be achieved by projecting into the $\mathbf{x}_1$, $\mathbf{x}_2$ coordinate plane.
If the selected coordinates produce an immersion, then the locally linear approximation
\begin{equation}
    \mathbf{x} = \mathbf{\bar{x}} + \mathbf{U}_{\mathbf{\bar{x}}}\left(\boldsymbol{\Pi}_{\mathscr{P}}^T\mathbf{U}_{\mathbf{\bar{x}}}\right)^{+} \boldsymbol{\Pi}_{\mathscr{P}}^T(\mathbf{x} - \mathbf{\bar{x}}) + \mathcal{O}\left(\Vert \mathbf{x} - \mathbf{\bar{x}} \Vert^2 \right) \quad \mbox{as} \quad \mathbf{x}\to\mathbf{\bar{x}}, \quad
    \mathbf{x}\in\mathcal{U}_{\mathbf{\bar{x}}}\subseteq\mathcal{M},
    \label{eqn:locallyLinearRecon0}
\end{equation}
holds for all points on the manifold.
Compactness of $\mathcal{M}$ says that there is some finite collection of these neighborhoods $\lbrace \mathcal{U}_1,\ldots,\mathcal{U}_N \rbrace \subseteq \lbrace \mathcal{U}_{\mathbf{\bar{x}}} \rbrace_{\mathbf{\bar{x}}\in\mathcal{M}}$ that cover $\mathcal{M}$.
Since the map $\boldsymbol{\Pi}_{\mathscr{P}}^T$ is 1-1 on each of these neighborhoods, we can think of $\boldsymbol{\Pi}_{\mathscr{P}}^T$ as a local embedding on a finite collection of patches that cover $\mathcal{M}$.
This means that all local ambiguities between points in neighborhoods $\mathcal{U}_{i}$, $i=1\ldots,N$ on $\mathcal{M}$ are resolved by measuring the coordinate indices $\mathscr{P}$.
For example, a finite collection of such patches on which projection into the $\mathbf{x}_1$, $\mathbf{x}_2$ coordinate plane is a local embedding are shown in \cref{fig:SpiralMfd_ImmersionDiagram_Nhds}.
If we can take $\mathcal{U}_1 = \mathcal{U}_{\mathbf{\bar{x}}} = \mathcal{M}$ for any $\mathbf{\bar{x}}\in\mathcal{M}$, then $\boldsymbol{\Pi}_{\mathscr{P}}^T$ is an embedding of $\mathcal{M}$.
\begin{figure}
    \centering
    \subfloat[]{
    \includegraphics[width=0.48\linewidth]{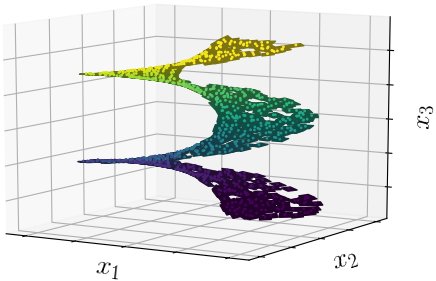}
    \label{fig:SpiralMfd_ImmersionDiagram_Planes}
    }
    \hfill
    \subfloat[]{
    \includegraphics[width=0.48\linewidth]{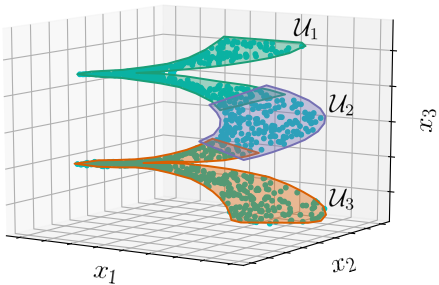}
    \label{fig:SpiralMfd_ImmersionDiagram_Nhds}
    }
    \caption{Synthetic data points randomly sampled on the spiral shaped manifold described by \cref{eqn:SpiralMfd} in \cref{subapp:SpiralMfd}.
    Observe that the underlying manifold can be immersed in coordinates $\mathbf{x}_1$ and $\mathbf{x}_2$, but we need the additional coordinate $\mathbf{x}_3$ to produce an embedding. \Cref{fig:SpiralMfd_ImmersionDiagram_Planes} shows tangent planes computed at each point and \Cref{fig:SpiralMfd_ImmersionDiagram_Nhds} shows a finite collection of neighborhoods over which the projection into $\mathbf{x}_1$ and $\mathbf{x}_2$ is 1-1.}
    \label{fig:SpiralMfd_ImmersionDiagram}
\end{figure}

\subsection{Simultaneously Pivoted QR Method}
\label{subsec:simPQRintro}
Since the manifold is compact, let us construct a representative collection of orthonormal matrices $\mathbf{U}_1,$ $\ldots,$ $\mathbf{U}_K\in\mathbb{R}^{n\times r}$ whose columns span the tangent spaces $T_{\mathbf{\bar{x}}_k}\mathcal{M}$ at a set of points $\mathbf{\bar{x}}_1,$ $\ldots,$ $\mathbf{\bar{x}}_K\in\mathcal{M}$.
These points are chosen so that any point $\mathbf{x}\in\mathcal{M}$ may closely estimated by the locally linear approximation \cref{eqn:locallyLinearRecon0} about some $\mathbf{\bar{x}}_k\in\mathcal{M}$.
These tangent spaces, the associated bases, and locally linear approximations will often be referred to as ``patches'' for short throughout this paper.
We will call $\Vol(\boldsymbol{\Pi}_{\mathscr{P}}^T\mathbf{U}_{\mathbf{\bar{x}}})$ the ``volume'' of the patch at $\mathbf{\bar{x}}\in\mathcal{M}$ when coordinates $\mathscr{P}$ are selected.

\begin{remark}[Getting Tangent Spaces from Data]
\label{rem:TangentSpacesFromData}
There are a variety of methods that can be used to construct or approximate the matrices $\mathbf{U}_1,$ $\ldots,$ $\mathbf{U}_K\in\mathbb{R}^{n\times r}$ from training data points sampled on $\mathcal{M}$.
For example, we may have $\mathcal{M} = \lbrace \mathbf{x}(\boldsymbol{\xi})\in\mathbb{R}^n\ :\ \xi\in\Omega\subset\mathbb{R}^r \rbrace$, where $\mathbf{x}(\boldsymbol{\xi})$ is the result of an experiment that depends one-to-one and smoothly on the parameters $\boldsymbol{\xi}$.
If experiments are performed at a collection of non-co-planar parameter values $\boldsymbol{\xi}_1,$ $\ldots,$ $\boldsymbol{\xi}_m$, then an interpolating function $\mathbf{\tilde{x}}(\boldsymbol{\xi})$ can be constructed.
The interpolating function is differentiated $\frac{\partial}{\partial\xi_1}\mathbf{\tilde{x}}(\boldsymbol{\bar{\xi}}),$ $\ldots,$ $\frac{\partial}{\partial\xi_r}\mathbf{\tilde{x}}(\boldsymbol{\bar{\xi}})$, yielding a collection of vectors spanning an approximation to the tangent space at $\mathbf{x}(\boldsymbol{\bar{\xi}})\in\mathcal{M}$.

Otherwise, if the underlying parameters $\boldsymbol{\xi}$ are not known, but we have a collection of training data from $\mathcal{M}$, then we may use Laplacian eigenmaps \cite{Belkin2003} or other manifold learning techniques to find a set of parameterizing coordinates.
The method described in \cref{app:LapEigAndTanSpace} shows how to interpolate and differentiate the Laplacian eigenmaps coordinates in order to construct vectors that span the tangent spaces at training data points.
An even simpler alternative might be to find a collection of nearest neighbors to points $\mathbf{\bar{x}}_1,$ $\ldots,$ $\mathbf{\bar{x}}_K\in\mathcal{M}$ and perform local principal component analysis (PCA). 
This yields a basis for a secant approximation to the tangent space.
\end{remark}

Once a representative collection of patches has been found, we shall construct a ``Simultaneously Pivoted'' QR (SimPQR) factorization
\begin{equation}
\begin{bmatrix}
\mathbf{U}_1^T \\
\vdots \\
\mathbf{U}_K^T
\end{bmatrix}
\begin{bmatrix}[c|c]
\boldsymbol{\Pi}_{\mathscr{P}} & \boldsymbol{\Pi}_{\mathscr{P}^c} \end{bmatrix} = 
\begin{bmatrix}
\mathbf{Q}_1 & & \\
& \ddots & \\
& & \mathbf{Q}_K
\end{bmatrix}
\begin{bmatrix}[c|c]
\mathbf{R}_1^{(1)} & \mathbf{R}_1^{(2)} \\
\vdots & \vdots \\
\mathbf{R}_K^{(1)} & \mathbf{R}_K^{(2)}
\end{bmatrix},
\label{eqn:SimPQR}
\end{equation}
in which the pivot columns are shared across the patches $\mathbf{U}_1^T,$ $\ldots,$ $\mathbf{U}_K^T$.
The matrices $\mathbf{R}_k^{(1)}\in\mathbb{R}^{r\times d}$ will each contain an $r\times r$ upper-triangular sub-matrix $\mathbf{\tilde{R}}_k^{(1)}$ corresponding to local coordinate indices $\mathscr{P}_k\subseteq\mathscr{P}$.
We shall try to pick pivot columns that make all volumes $\det(\mathbf{\tilde{R}}_k^{(1)})$ simultaneously large, ensuring that all $\boldsymbol{\Pi}_{\mathscr{P}}^T\mathbf{U}_k$ are left-invertible, with volumes lower-bounded by 
\begin{equation}
    \Vol{\left(\boldsymbol{\Pi}_{\mathscr{P}}^T\mathbf{U}_k\right)} 
    \geq \det{\left(\mathbf{\tilde{R}}_k^{(1)}\right)} 
    = \prod_{i=1}^r \vert \mathbf{\tilde{R}}_k^{(1)}(i,i) \vert.
\end{equation}
If we let $\boldsymbol{\Pi}_{\mathscr{P}_k}$ and $\boldsymbol{\Pi}_{\mathscr{P}_k^c}$ be the permutation matrices associated with the patch-wise indices $\mathscr{P}_k$ and $\mathscr{P}_k^c$ respectively, then a local PQR factorization
\begin{equation}
\mathbf{U}_k^T 
\begin{bmatrix}[c|c]
\boldsymbol{\Pi}_{\mathscr{P}_k} & \boldsymbol{\Pi}_{\mathscr{P}_k^c}
\end{bmatrix} = 
\mathbf{Q}_k
\begin{bmatrix}[c|c]
\mathbf{\tilde{R}}_k^{(1)} & \mathbf{\tilde{R}}_k^{(2)}
\end{bmatrix},
\end{equation}
is obtained on each patch from the SimPQR factorization.
Therefore, the SimPQR factorization \cref{eqn:SimPQR} is analogous to PQR factorization \cref{eqn:PQR} applied simultaneously with each $\mathscr{P}_k\subseteq\mathscr{P}$ to bases coming from multiple linear approximations of the manifold.

In order to develop an algorithm for computing a SimPQR factorization \cref{eqn:SimPQR}, let us consider two new facets of the immersion problem for collections of patches that arise in the nonlinear setting.
First, observe that when the underlying manifold is curved, there will be an inherent trade-off between choosing fewer total coordinates $\mathscr{P} = \bigcup_{k=1}^K \mathscr{P}_k$, and choosing local coordinates $\mathscr{P}_k$ that maximize $\Vol{\left(\boldsymbol{\Pi}_{\mathscr{P}_k}^T\mathbf{U}_k\right)}$.
On one hand, we could force the $\mathscr{P}_k$'s to all overlap as much as possible, minimizing $\vert \mathscr{P} \vert$ and creating a parsimonious immersion of $\mathcal{M}$.
However, different sets of coordinates $\mathscr{P}_k$ may achieve high volumes $\Vol{\left(\boldsymbol{\Pi}_{\mathscr{P}_k}^T\mathbf{U}_k\right)}$ on different patches, so forcing the $\mathscr{P}_k$'s to overlap may come at the cost of smaller patch volumes.
On the other hand, we could pick each $\mathscr{P}_k$, possibly non-overlapping, to maximize $\Vol{\left(\boldsymbol{\Pi}_{\mathscr{P}_k}^T\mathbf{U}_k\right)}$ thereby creating an immersion that allows for highly noise-robust reconstruction using \cref{eqn:locallyLinearRecon0} at the sacrifice of parsimony.
This second case could be approximately accomplished by performing PQR separately on each patch $\mathbf{U}_k^T$ to find $\mathscr{P}_k$ and combining all the coordinates into $\mathscr{P} = \bigcup_{k=1}^K \mathscr{P}_k$ at the end.
As we will discuss in \cref{subsec:theoreticalGuarantees}, our proposed \cref{alg:SimPQR} for computing a SimPQR factorization provably balances parsimony and robustness at each stage through a user-defined parameter $\gamma$ (see equation \cref{eqn:SimPQR_optimalPivoting}).

Much like \cref{alg:PQR} for computing a PQR factorization, our proposed \cref{alg:SimPQR} will use a column-pivoted Householder QR process to compute a SimPQR factorization.
However, on curved manifolds, we can't always use the selected pivot column to perform Householder reflections on every patch.
In other words, the geometry of $\mathcal{M}\subset\mathbb{R}^n$ may only allow us to add a pivot column $j^*$ to some, but not all sets $\mathscr{P}_k$ during a given stage, resulting in $\mathscr{P}_1$, $\ldots$, $\mathscr{P}_K$ not all equal.
For example, a circle in $\mathbb{R}^2$ is a $1$-dimensional manifold.
Yet, if we choose either coordinate, $x_1$ or $x_2$, as the first pivot column, there will be a tangent vector whose entry at the chosen coordinate is $0$.
There will be many more tangent vectors whose entries at the chosen coordinate are very small.
Such entries should not be used to construct Householder reflections on those patches since the result will be $\boldsymbol{\Pi}_{\mathscr{P}_k}$-singular or nearly $\boldsymbol{\Pi}_{\mathscr{P}_k}$-singular.
Instead, we should use the selected coordinate for Householder QR only on the patches where it will yield an acceptably high volume $\Vol{\left(\boldsymbol{\Pi}_{\mathscr{P}_k}^T\mathbf{U}_k\right)}$.
Then, we should choose the other coordinate to give an acceptably high volumes on the remaining patches.

Another example is the $2$-dimensional cylindrical manifold embedded in $\mathbb{R}^{3}$ discussed in \cref{subapp:cylMfd} whose tangent planes at $K=1000$ randomly sampled points are shown in \cref{fig:CylMfd}.
All tangent planes have a component along the $x_3$ axis, but some lie orthogonal or nearly orthogonal to the $x_1$ or $x_2$ axes.
Thus, the choice of any two coordinates will not be sufficient to resolve local ambiguities between all sets of nearby points.
When finding a SimPQR factorization, if the $x_1$ coordinate is chosen, then Householder QR cannot be carried out on those patches whose tangent planes are nearly orthogonal to the $x_1$ axis.
\textit{Therefore, at each stage of SimPQR we must choose both the pivot column and the patches over which it can be used to perform Householder transformations.} 
The local coordinates sets $\mathscr{P}_k$ keep track of which coordinates were used to perform Householder reflections on patch $k$.
Additionally, it may also be beneficial in some cases to choose the same coordinate again at a later stage for use on a different collection of patches, so we must leave all coordinates available for selection during every stage.
The above considerations highlight aspects of our proposed SimPQR factorization that are needed in the nonlinear setting and distinguish it from standard PQR.

\begin{figure}
    \centering
    \includegraphics[width=0.6\linewidth]{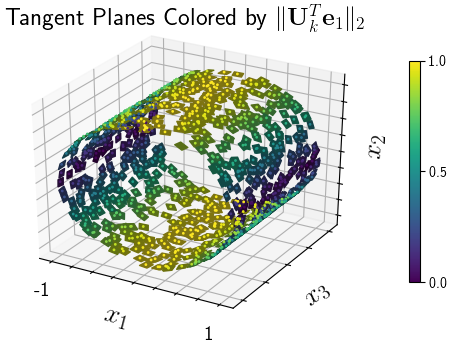}
    \caption{Tangent planes to a cylindrical manifold at $K=1000$ points. The points are colored by the magnitude of the tangent space basis vectors along the $x_1$ axis. We see that the basis vectors are both nearly zero in their $x_1$ entry when $\mathbf{x}_1\approx \pm 1.0$.}
    \label{fig:CylMfd}
\end{figure}

To compute a SimPQR factorization with column-pivoted Householder reflections as in \cref{alg:SimPQR}, let us initialize the residual matrices $\mathbf{A}_k = \mathbf{U}_k^T$ and column magnitudes $\mathbf{c}_k(j) = \Vert \mathbf{A}_k(:,j) \Vert_2^2,$ $j=1,\ldots,n$ on every patch $k=1$, $\ldots$, $K$; they will be updated during each stage of the algorithm.
Let $\mathscr{S}_j = \lbrace k:\ \mathbf{c}_k(j) > \epsilon \rbrace$ denote the set of patches whose $j$th coordinate is not yet eliminated by the factorization process and let
\begin{equation}
    S_{max} = \max_{1\leq j \leq n} \vert \mathscr{S}_j \vert
\end{equation}
be the greatest number of such patches across the choices of coordinates.
In other words, $S_{max}$ is the greatest number of patches on which SimPQR could simultaneously perform Householder reflections at the given stage.
Let the maximum column magnitudes over all the coordinates on each patch be given by
\begin{equation}
    C_{k} = \max_{1\leq j\leq n} \mathbf{c}_k(j).
\end{equation}
These are the magnitudes of the columns that would be selected by Businger-Golub pivoting on each patch.
At each stage, we shall select the pivot column (coordinate index) $j^*$ and the patch index set $\mathscr{S}^*$ over which to perform Householder QR according to
\begin{equation}
    \boxed{
    (j^*, \mathscr{S}^*) = 
    \argmax_{1\leq j\leq n,\ \mathscr{S}\subseteq\mathscr{S}_j}
    \left(
    \frac{\vert \mathscr{S} \vert}{S_{max}} + 
    \gamma\cdot \min_{k\in\mathscr{S}} \sqrt{\frac{\mathbf{c}_k(j)}{C_{k}}}
    \right).
    }
    \label{eqn:SimPQR_optimalPivoting}
\end{equation}
\textit{We emphasize that the pivot column $j^*$ is used to perform Householder reflections only on the collection of patches indexed by $\mathscr{S}^*$ during a given stage of SimPQR.}
Consequently, the local coordinates sets $\mathscr{P}_k$ are updated according to
\begin{equation}
    \mathscr{P}_k \leftarrow \mathscr{P}_k \cup \lbrace j^* \rbrace, \qquad \forall k\in\mathscr{S}^*.
\end{equation}
We see intuitively from \cref{eqn:SimPQR_optimalPivoting} that during each stage, \textit{the user-defined parameter $\gamma$ trades off between selecting the maximum number of patches $S_{max}$ at one extreme $\gamma\to 0$ and achieving the maximum patch volume by picking the largest column on each patch at the other extreme $\gamma \to \infty$.}
The process terminates when all patches have been factored.

The patch-wise index sets $\mathscr{P}_k$ and $\mathscr{P}_k^c$ keep track of the coordinate indices that have been selected and remain to be selected respectively on each patch.
Once the factorization is completed, the shared index lists
\begin{equation}
    \label{eqn:shared_index_lists}
    \mathscr{P} = \bigcup_{k=1}^K \mathscr{P}_k, \qquad \mbox{and} \qquad \mathscr{P}^c = \bigcap_{k=1}^K \mathscr{P}_k^c,
\end{equation}
are formed, where $\mathscr{P}$ contains all of the selected coordinate indices and $\mathscr{P}^c$ contains the remaining unselected coordinate indices shared across all patches.
\Cref{alg:SimPQR} describes this method for computing a SimPQR factorization.
It makes reference to \cref{alg:optimalPivoting} in order to solve the optimization problem \cref{eqn:SimPQR_optimalPivoting} which is examined next in \cref{subsec:optimalPivoting}.
\begin{algorithm}
  \caption{Simultaneously Pivoted Householder QR (SimPQR)}
  \label{alg:SimPQR}
  \begin{algorithmic}
    \STATE{Initialize: $\mathbf{A}_k = \mathbf{U}_k^T$ for $k=1,\ldots,K$}
    \COMMENT{matrix to be over-written}
    \STATE{Initialize: $\mathscr{P}_k = []$, $\mathscr{P}_k^c = [1,\ldots,n]$ for $k=1,\ldots,K$}
    \COMMENT{patch-wise pivot index lists}
    \STATE{Initialize: $i_1 = 1$, $\ldots,$ $i_K = 1$}
    \COMMENT{QR progress on each patch}
    \STATE{$\mathbf{c}_k(j) = \Vert \mathbf{A}_k(:,j) \Vert_2^2$ for all $k=1,\ldots,K$ and $j=1,\ldots,n$}
    \COMMENT{residual column magnitudes}
    \WHILE[proceed while not all patches have been factored]{$\exists k,j$ s.t. $\mathbf{c}_k(j) > \epsilon$}
        \STATE{Compute pivot column $j^*$ and patch index set $\mathscr{S}^*$ using \cref{alg:optimalPivoting}}
        \FORALL[QR on selected patches]{$k\in\mathscr{S}^*$}
            \STATE{$(\mathbf{v},\ \beta) = \house (\mathbf{A}_k(i_k:r,j^*))$}
            \COMMENT{Householder transform to zero entries in piv col}
            \STATE{$\mathbf{A}_k(i_k:r, \mathscr{P}_k^c) \leftarrow (\mathbf{I}_{r-i_k+1} - \beta \mathbf{v}\mathbf{v}^T) \mathbf{A}_k(i_k:r, \mathscr{P}_k^c)$}
            \COMMENT{transform remaining columns}
            \STATE{$\mathbf{A}_k(i_k+1:r,j^*) \leftarrow \mathbf{v}(2:r-i_k+1)$}
            \COMMENT{store vector in zeroed entries of piv col}
            \FOR[loop over remaining columns]{$j\in\mathscr{P}_k^c$}
                \STATE{$\mathbf{c}_k(j) \leftarrow \mathbf{c}_k(j) - \mathbf{A}_k(i_k,j)^2$}
                \COMMENT{update residual magnitudes}
            \ENDFOR
            \STATE{$\mathscr{P}_k \leftarrow \mathscr{P}_k\cup [j^*]$ and $\mathscr{P}_k^c \leftarrow \mathscr{P}_k^c\setminus [j^*]$}
            \COMMENT{update pivot index lists}
            \STATE{$i_k \leftarrow i_k + 1$}
            \COMMENT{update local QR progress}
        \ENDFOR
    \ENDWHILE
    \STATE{$\mathscr{P} = \bigcup_{k=1}^K \mathscr{P}_k$ and $\mathscr{P}^c = \bigcap_{k=1}^K \mathscr{P}_k^c$}
    \COMMENT{form joint index sets}
    \RETURN{$\mathscr{P}$, $\mathscr{P}^c$, and $\mathscr{P}_k$, $\mathscr{P}_k^c$, $\mathbf{A}_k$ for $k=1,\ldots,K$}
  \end{algorithmic}
\end{algorithm}

\begin{remark}
For a description of how $\mathbf{Q}_k$ and $\mathbf{R}_k = \begin{bmatrix}[c|c]
\mathbf{R}_k^{(1)} & \mathbf{R}_k^{(2)}
\end{bmatrix}$ can be obtained from the output of \cref{alg:SimPQR}, see \cref{app:Householder}.
\end{remark}

\subsection{Optimal Coordinate and Patch Selection}
\label{subsec:optimalPivoting}
The optimization problem in \cref{eqn:SimPQR_optimalPivoting} appears nasty at first, but it is actually easy to solve and offers several theoretical guarantees discussed in \cref{subsec:theoreticalGuarantees}.
Let us now consider how to make the optimal selection $j^*$ and $\mathscr{S}^*$ according to \cref{eqn:SimPQR_optimalPivoting}.
Computing $S_{max}$ and $\lbrace C_{k} \rbrace_{k=1}^K$ requires looping once over all $K$ patches for each of the $n$ coordinates; hence $\mathcal{O}(nK)$ operations are performed to compute them.
Now, at fixed $j$ we may solve the optimization problem \cref{eqn:SimPQR_optimalPivoting} over the patch index set $\mathscr{S}$ by sorting the normalized candidate column magnitudes $\boldsymbol{\Theta}(k,j) = \sqrt{\mathbf{c}_k(j) / C_k}$ from greatest to least according to
\begin{equation}
    \boldsymbol{\Theta}(\mathbf{k}_j(1),j) \geq 
    \boldsymbol{\Theta}(\mathbf{k}_j(2),j) \geq \cdots \geq
    \boldsymbol{\Theta}(\mathbf{k}_j(K),j) \geq 0.
\end{equation}
This requires $\mathcal{O}(K\log{K})$ operations per coordinate.
Choosing candidate patch index lists $\mathscr{S}_{l,j} = \lbrace \mathbf{k}_j(1),$ $\ldots,$ $\mathbf{k}_j(l)\rbrace,$ $l=1,\ldots,K$ gives a monotonically decreasing sequence
\begin{equation}
    \left\lbrace \min_{k\in\mathscr{S}_{l,j}} \sqrt{\frac{ \mathbf{c}_k(j)}{C_k}} = \boldsymbol{\Theta}(\mathbf{k}_j(l),j)\right\rbrace_{l=1}^K,
\end{equation}
attaining all possible values of $\min_{k\in\mathscr{S}} \sqrt{\frac{ \mathbf{c}_k(j)}{C_k}}$ across patch index lists $\mathscr{S}$ at fixed $j$.
Furthermore, if $\boldsymbol{\Theta}(\mathbf{k}_j(l),j) > \boldsymbol{\Theta}(\mathbf{k}_j(l+1),j)$ then $\mathscr{S}_{l,j}$ is the largest list of patch indices attaining the normalized column magnitude $\boldsymbol{\Theta}(\mathbf{k}_j(l),j)$ at fixed $j$.
It follows that the optimal list of patch indices $\mathscr{S}_j^*$ at the coordinate $j$ will be among the $\mathscr{S}_{1,j},$ $\ldots,$ $\mathscr{S}_{K,j}$.
We therefore find the optimal patch list index $l^*$ and coordinate index $j^*$ by maximizing the objective in \cref{eqn:SimPQR_optimalPivoting} according to 
\begin{equation}
    \label{eqn:sorted_SimPQR_objective}
    (l^*, j^*) = \argmax_{1\leq l\leq K,\ 1\leq j\leq n:\ \mathbf{c}_{k_{l,j}}(j) > \epsilon} \left( \frac{l}{S_{max}} + \gamma\boldsymbol{\Theta}(\mathbf{k}_j(l), j) \right).
\end{equation}
The optimal list of patch indices using the optimal pivot coordinate is $\mathscr{S}^* = \mathscr{S}_{l^*,j^*}$.
The above is a direct search over all patch-coordinate pairs requiring $\mathcal{O}(n K)$ operations.

The complete process described above for solving \cref{eqn:SimPQR_optimalPivoting} at each stage is given in \cref{alg:optimalPivoting} and requires at total of $\mathcal{O}(nK\log K)$ operations and $\mathcal{O}(n + K)$ additional space.
The Householder QR process at each stage requires $\mathcal{O}( n \vert \mathscr{S}^* \vert r )$ operations, hence,
\begin{equation}
    \mathcal{O}(nK\log K + n \vert \mathscr{S}^* \vert r)
\end{equation}
operations are performed at each stage.
The total number of stages may be as small as $r$ when the chosen coordinates can be shared across all patches or as large as $rK$ when choosing separate coordinates on each patch.
The choice of $\gamma$ will determine the number of stages in SimPQR.
If $\mathbf{U}_k^T$ can be over-written, then the space complexity of performing SimPQR is $\mathcal{O}( n K )$ since it is dominated by storing $\lbrace \mathbf{c}_k \rbrace_{k=1}^K$, otherwise there is an additional factor of $r$ since $\lbrace \mathbf{A}_k \rbrace_{k=1}^K$ must be stored.

\begin{algorithm}
  \caption{Optimal Coordinate and Patch Selection for SimPQR}
  \label{alg:optimalPivoting}
  \begin{algorithmic}
    \STATE{Given: $\mathbf{c}_k$ for $k=1,\ldots,K$}
    \COMMENT{current patch volumes and residuals}
    \STATE{Initialize: $C_{1},\ldots,C_{k}=0$}
    \COMMENT{maximum column magnitudes per patch}
    \STATE{Initialize: $S_{1},\ldots,S_{n}=0$ and $S_{max} = 0$}
    \COMMENT{nonzero patch counts per coordinate}
    \FOR[loop over coordinates]{$j=1,\ldots,n$}
        \FOR[loop over patches]{k=1,\ldots,K}
            \STATE{$C_{k} \leftarrow \max\lbrace \mathbf{c}_k(j),\ C_k \rbrace$}
            \COMMENT{update max column magnitudes}
            \STATE{$S_{j} \leftarrow S_{j} + \mathbbm{1}\lbrace \mathbf{c}_k(j) > \epsilon \rbrace$}
            \COMMENT{update nonzero patch counts}
        \ENDFOR
        \STATE{$S_{max} \leftarrow \max\lbrace S_{j}, S_{max} \rbrace$}
        \COMMENT{update max nonzero patch count}
    \ENDFOR
    \STATE{Initialize: $(J^*,\ j^*,\ \mathscr{S}^*) = (0,\ 0,\ \emptyset)$}
    \FOR[optimize over coordinates]{$j=1,\ldots,n$}
        \STATE{$\mathscr{K} = \argsortDown_{k=1,\ldots,K} (\mathbf{c}_k(j) / C_k)$}
        \COMMENT{sort patch magnitudes from greatest to least}
        \STATE{Initialize: $(J_{col}^*,\ l_{col}^*) = (0,\ 0)$}
        \COMMENT{optimal objective and $\mathscr{K}$ index at fixed $j$}
        \FOR[optimize over sorted patch indices]{$l=1,\ldots,S_j$}
            \STATE{$k = \mathscr{K}(l)$}
            \COMMENT{$l$th sorted patch index}
            \STATE{$J = l/S_{max} + \gamma\cdot \sqrt{\mathbf{c}_k(j)/C_k}$}
            \COMMENT{objective value}
            \IF[update optimal objective and $\mathscr{K}$ index at $j$]{$J > J_{col}^*$} 
                \STATE{$(J_{col}^*,\ l_{col}^*)\leftarrow (J,\ l)$}
            \ENDIF
        \ENDFOR
        \IF[update optimal objective, coordinate, patch set]{$J_{col}^* > J^*$}
            \STATE{$(J^*,\ j^*,\ \mathscr{S}^*) \leftarrow (J_{col}^*,\ j,\ \lbrace \mathscr{K}(1),\ldots,\mathscr{K}(l_{col}^*) \rbrace )$}
        \ENDIF
    \ENDFOR
    \RETURN{$j^*$ and $\mathscr{S}^*$}
    \COMMENT{optimal coordinate and patch index set}
  \end{algorithmic}
\end{algorithm}

\subsection{Parameterized Solution Path for SimPQR}
\label{subsec:GammaPath}
In many applications, we may not know the best value for the parameter $\gamma$ ahead of time.
In these cases it is useful to construct solutions sweeping over a range of values for $\gamma$ and compare the models using cross-validation \cite{Hastie2009}.
Luckily, the SimPQR factorization undergoes a finite number changes as the parameter ranges over $(0, \infty)$.
Furthermore, given a value of the parameter $\gamma$ it is possible to find the closest neighboring values $\gamma^{-} < \gamma$ and $\gamma^{+} > \gamma$ that lead to different sequences pivot and patch indices chosen by optimizing \cref{eqn:SimPQR_optimalPivoting} at each stage of SimPQR.
These neighboring values of the parameter can be found without increasing the order-wise operation count of SimPQR, \cref{alg:SimPQR}.

Recall that the optimal coordinate index $j^*$ and patch index set $\mathscr{S}^* = \mathscr{S}_{l^*, j^*}$ are found by maximizing the objective \cref{eqn:sorted_SimPQR_objective},
\begin{equation}
    (l^*, j^*) = \argmax_{1\leq l\leq K,\ 1\leq j\leq n:\ \mathbf{c}_{\mathbf{k}_j(l)}(j) > \epsilon} J_{l,j}(\gamma), \qquad
    J_{l,j}(\gamma) \triangleq \underbrace{\frac{l}{S_{max}}}_{J^{(1)}_{l,j}} + \gamma\underbrace{\boldsymbol{\Theta}(\mathbf{k}_j(l), j)}_{J^{(2)}_{l,j}},
\end{equation}
over the coordinate indices and sorted patches where $\mathbf{c}_{\mathbf{k}_j(l)}(j) > \epsilon$.
Therefore, for each pair $(l,j)$ with $\mathbf{c}_{\mathbf{k}_j(l)}(j) > \epsilon$ and $J^{(2)}_{l,j} \neq J^{(2)}_{l^*,j^*}$, there exists
\begin{equation}
    \gamma_{l,j} = \frac{J^{(1)}_{l^*,j^*} - J^{(1)}_{l,j}}{J^{(2)}_{l,j} - J^{(2)}_{l^*,j^*}},
\end{equation}
where $J_{l,j}(\gamma_{l,j}) = J_{l^*,j^*}(\gamma_{l,j})$.
During a given stage $i$ of SimPQR,
\begin{equation}
\begin{aligned}
    \gamma^{-}_i &\triangleq \max \lbrace \gamma_{l,j}:\  \gamma_{l,j} < \gamma,\ \mathbf{c}_{\mathbf{k}_j(l)}(j) > \epsilon,\ \mbox{and}\ J^{(2)}_{l,j} \neq J^{(2)}_{l^*,j^*}\rbrace \cup \lbrace -\infty \rbrace \\
    \gamma^{+}_i &\triangleq \min \lbrace \gamma_{l,j}:\  \gamma_{l,j} > \gamma,\ \mathbf{c}_{\mathbf{k}_j(l)}(j) > \epsilon,\ \mbox{and}\ J^{(2)}_{l,j} \neq J^{(2)}_{l^*,j^*}\rbrace \cup \lbrace \infty \rbrace
\end{aligned}
\end{equation}
are the neighboring values of $\gamma$ where taking $\gamma > \gamma^{+}_i$ or $\gamma < \gamma^{-}_i$ will cause a different value of $l^*$ or $j^*$ to be chosen by maximizing \cref{eqn:sorted_SimPQR_objective}.
If SimPQR terminates after $N$ stages, then the neighboring values of the parameter below or above which a change in the selected ordering of coordinates or patches occurs are
\begin{equation}
    \gamma^{-} = \max_{1\leq i\leq N} \gamma^{-}_i \qquad \mbox{and}\qquad
    \gamma^{+} = \min_{1\leq i\leq N} \gamma^{+}_i.
\end{equation}

As will be shown in \cref{thm:gamma_simul}, the smallest value of $\gamma$ that we need to consider is $1/K$. 
Therefore, the entire solution path for SimPQR over the values of $\gamma$ can be found by starting at $\gamma = 1/K$ and computing $\gamma^{+}$ during SimPQR, then performing SimPQR at $\gamma = \gamma^{+} + \delta$, where $\delta$ is sufficiently small, and repeating until $\gamma^+ = \infty$.

\subsection{Theoretical Guarantees for SimPQR}
\label{subsec:theoreticalGuarantees}
The objective function in \cref{eqn:SimPQR_optimalPivoting} uses the parameter $\gamma$ to balance the normalized number of patches $\vert \mathscr{S} \vert / S_{max}$ sharing the $j$th coordinate against the minimum normalized column magnitude $\min_{k\in\mathscr{S}} \sqrt{\mathbf{c}_k(j) / C_k}$ across those patches.
Recall that on each patch the total volume $V_k = \det (\mathbf{\tilde{R}}_k^{(1)})$ is the product of the pivot column magnitudes $\sqrt{\mathbf{c}_k(j^*)}$ at each stage; hence it can be computed by letting $V_k=1$ and updating $V_k \leftarrow V_k \cdot \sqrt{\mathbf{c}_k(j^*)}$ for every $k\in\mathscr{S}^*$ at each stage.
Since maximizing the patch volume forces the singular values of $\boldsymbol{\Pi}_{\mathscr{P}}^T\mathbf{U}_k$ to be large, it is a surrogate for the robustness of the linear reconstruction \cref{eqn:locallyLinearRecon0} to perturbations in the selected states.
\Cref{thm:gamma_bound} shows that when the parameter $\gamma\in(1-1/K,\infty)$, it is directly related to the smallest fraction $\eta$ of maximum volume SimPQR is allowed to attain at each stage in order to share the chosen coordinate across multiple patches.
Allowing the fraction $\eta$ of the maximum volume to become smaller allows SimPQR to share the same coordinate over more patches, thereby choosing fewer coordinates in total.
However, as $\eta\to 0^+$ and $\gamma \to (1-1/K)^+$, the patch volumes may become small and the reconstruction \cref{eqn:locallyLinearRecon0} may not be as robust against perturbations in the selected coordinates.
\begin{theorem}[Parameter Choice and Robustness]
\label{thm:gamma_bound}
Taking $\gamma \geq \frac{K-1}{K(1-\eta)}$ ensures that a minimum column magnitude $\sqrt{\mathbf{c}_k(j^*)} \geq \eta \sqrt{C_k}$ relative to a Businger-Golub choice is achieved for all $k\in\mathscr{S}^*$.
\end{theorem}
\begin{proof}
See \cref{app:Proofs}
\end{proof}

In the case where $\gamma\in(0,1)$, \cref{thm:gamma_simul} puts a lower bound on the number of patches that SimPQR selects simultaneously.
In fact, choosing $\gamma \leq 1/K$ ensures that the algorithm always chooses the maximum number of non-singular patches $S_{max}$ at every stage, thereby minimizing the total number of selected coordinates.
Ties are broken by picking the coordinate the yields the greatest minimum fraction of the volume attained by the Businger-Golub pivot over the patches.
\begin{theorem}[Parameter Choice and Simultaneity]
\label{thm:gamma_simul}
Taking $\gamma \leq 1-\nu$ ensures that the number of patches chosen simultaneously is at least $\vert \mathscr{S}^* \vert > \nu S_{max}$.
Furthermore, taking $\gamma \leq 1/K$ ensures that SimPQR always chooses $\vert \mathscr{S}^* \vert = S_{max}$ patches at each stage.
\end{theorem}
\begin{proof}
See \cref{app:Proofs}
\end{proof}

One may wonder if SimPQR is related to PQR as $\gamma$ becomes large.
Indeed, as $\gamma \to \infty$, \cref{thm:PQR_equivalence} shows that SimPQR eventually becomes equivalent to performing PQR with a Businger-Golub pivoting performed separately on every patch.
This usually results in a large number of selected coordinates that enable maximally robust reconstruction of each patch in the same sense as PQR.
\textit{The user-defined parameter $\gamma>0$ therefore establishes the trade-off between the number of coordinates selected by the algorithm and the robustness of the local reconstruction.}
\begin{lemma}
\label{lem:maximumVolume}
At a given stage $N$ of SimPQR, there exists $\gamma_N < \infty$ so that $\gamma > \gamma_N$ ensures that $\mathbf{c}_k(j^*) = C_k$, hence $j^*$ is a Businger-Golub pivot column index, for all $k\in\mathscr{S}^*$.
\end{lemma}
\begin{proof}
See \cref{app:Proofs}
\end{proof}

\begin{theorem}[PQR Equivalence]
\label{thm:PQR_equivalence}
There exists a $\gamma_D < \infty$ so that taking $\gamma > \gamma_D$ recovers the same coordinates as performing a PQR factorization with Businger-Golub pivoting separately on every patch.
Note that the Businger-Golub pivot column is not unique if there is more than one column with the largest magnitude on a patch.
Taking $\gamma > \gamma_D$ recovers only one such choice at each stage.
\end{theorem}
\begin{proof}
See \cref{app:Proofs}
\end{proof}

\section{Discovering Embedding Coordinates for Nonlinear Manifolds}
\label{sec:BranchingCoordinates}
In \cref{sec:immersion_coordinates} we developed a technique for identifying primative features that locally parameterize the manifold $\mathcal{M}\subset\mathbb{R}^n$, thereby immersing it via coordinate projection.
The immersing coordinates eliminate the possibility of continuous ambiguities between manifold points by simultaneously paramterizing every tangent plane.
However, the global geometry of $\mathcal{M}$ can still result in an incomplete paramterization using these coordinates.
There may still be discrete ambiguities between points sharing the same immersion coordinates, but lying on different ``branches'' of the manifold.
For example, consider the spiral-shaped manifold discussed in \cref{subapp:SpiralMfd} and shown in \cref{fig:SpiralMfd_ImmersionDiagram}.
Applying the SimPQR algorithm to the tangent planes with any value of $\gamma\in(0,\infty)$ correctly identifies $x_1$ and $x_2$ as optimal immersion coordinates capable of locally parameterizing small patches of the manifold.
However, for fixed $x_1$ and $x_2$ there is a discrete collection of points with different values of $x_3$ that we refer to as lying on different ``branches'' of the manifold. 

In order to resolve these discrete ambiguities, we need to choose additional coordinates that allow for the multiple branches at any given set of immersing coordinates to be distinguished.
In order to make the notion of a ``branch'' concrete, we introduce the set of branches at a point as an equivalence class in the following proposition.
\begin{proposition}
\label{prop:finiteBranches}
If $\mathscr{P}_I$ is a set of immersion coordinate indices, then the equivalence classes
\begin{equation}
    [\mathbf{x}]_{\mathscr{P}_I} = \left\lbrace \mathbf{y}\in\mathcal{M}:\ \boldsymbol{\Pi}_{\mathscr{P}_I}^T(\mathbf{y} - \mathbf{x}) = \mathbf{0} \right\rbrace \in \mathcal{M}/\ker{\left(\boldsymbol{\Pi}_{\mathscr{P}_I}^T\right)}
\end{equation}
each contain a finite set of points lying on different ``branches'' of $\mathcal{M}$ in the coordinates $\mathscr{P}_I$.
\end{proposition}
\begin{proof}
See \cref{app:Proofs}.
\end{proof}
Note that if the projection into coordinates $\mathscr{P}_I$ is also an embedding, then $[\mathbf{x}]_{\mathscr{P}_I}$ contains exactly one point, $\mathbf{x}$.
Let
\begin{equation}
    \Delta[\mathbf{x}]_{\mathscr{P}_I} = \left\lbrace \mathbf{u} - \mathbf{v} :\  \mathbf{u},\mathbf{v}\in[\mathbf{x}]_{\mathscr{P}_I},\ \mathbf{u}\neq \mathbf{v} \right\rbrace \subset \ker{\left(\boldsymbol{\Pi}_{\mathscr{P}_I}^T\right)}
\end{equation}
be the set of all vectors separating distinct points in $[\mathbf{x}]_{\mathscr{P}_I}$.
This set is finite, by \cref{prop:finiteBranches}.
Suppose we choose additional sampling coordinates $\mathscr{P}_B\subseteq\mathscr{P}_I^c$ and form the corresponding sampling matrix $\boldsymbol{\Pi}_{\mathscr{P}_B}$ in order to distinguish between the different branches in $[\mathbf{x}]_{\mathscr{P}_I}$.
\Cref{prop:branchRecovery} says that it is possible to recover a specific $\mathbf{x}\in [\mathbf{x}]_{\mathscr{P}_I}$ from the sampled coordinates $\boldsymbol{\Pi}_{\mathscr{P}_B}^T\mathbf{x}$ as long as $\boldsymbol{\Pi}_{\mathscr{P}_B}^T\mathbf{w}$ is nonzero for all $\mathbf{w}\in\Delta[\mathbf{x}]_{\mathscr{P}_I}$.
In other words, if we have the sampled immersion coordinates $\boldsymbol{\Pi}_{\mathscr{P}_I}^T\mathbf{x}$ of an unknown point $\mathbf{x}\in\mathcal{M}$ then we can only determine the equivalence class $[\mathbf{x}]_{\mathscr{P}_I}$ to which it belongs.
If we sample the additional coordinates $\boldsymbol{\Pi}_{\mathscr{P}_B}^T\mathbf{x}$, then we can compare $\boldsymbol{\Pi}_{\mathscr{P}_B}^T\mathbf{x}$ to $\boldsymbol{\Pi}_{\mathscr{P}_B}^T\mathbf{y}$ for each $\mathbf{y}\in[\mathbf{x}]_{\mathscr{P}_I}$.
Since we obtain $\boldsymbol{\Pi}_{\mathscr{P}_B}^T\mathbf{x} = \boldsymbol{\Pi}_{\mathscr{P}_B}^T\mathbf{y}$ if and only if $\mathbf{x} = \mathbf{y}$, the point $\mathbf{x}\in\mathcal{M}$ can be determined uniquely from the coordinates $\boldsymbol{\Pi}_{\mathscr{P}_I}^T\mathbf{x}$ and $\boldsymbol{\Pi}_{\mathscr{P}_B}^T\mathbf{x}$.
Therefore, orthogonal projection into the combined set of coordinates $\mathscr{P}_E = \mathscr{P}_I\cup\mathscr{P}_B$ produces an embedding of $\mathcal{M}$.
\begin{proposition}[Branch Recovery]
\label{prop:branchRecovery}
If $\boldsymbol{\Pi}_{\mathscr{P}_B}^T\mathbf{w}$ is nonzero for all $\mathbf{w}\in\Delta[\mathbf{x}]_{\mathscr{P}_I}$ then
\begin{equation*}
    \mathbf{y} = \mathbf{z} \ \Longleftrightarrow \  
    \boldsymbol{\Pi}_{\mathscr{P}_B}^T(\mathbf{y} - \mathbf{z}) = \mathbf{0}, 
    \qquad \forall \mathbf{y},\mathbf{z}\in[\mathbf{x}]_{\mathscr{P}_I}.
\end{equation*}
\end{proposition}
\begin{proof}
If $\mathbf{y} = \mathbf{z}$ then $\boldsymbol{\Pi}_{\mathscr{P}_B}^T(\mathbf{y} - \mathbf{z}) = \mathbf{0}$ is trivial.
If $\mathbf{y} \neq \mathbf{z}$ then $\mathbf{y} - \mathbf{z}$ is an element of $\Delta[\mathbf{x}]_{\mathscr{P}_I}$.
It follows that $\boldsymbol{\Pi}_{\mathscr{P}_B}^T(\mathbf{y} - \mathbf{z}) \neq \mathbf{0}$.
\end{proof}

The size of vectors in $\Delta[\mathbf{x}]_{\mathscr{P}_I}$ upon projection into the branching coordinates $\mathscr{P}_B$ determines the maximum noise level below which branch recovery is possible.
\Cref{prop:branchRecoveryRobustness} shows that disturbances below this level do not change the nearest branch point in the sampled coordinates, allowing perturbed points lying near different branches to be distinguished.
\begin{proposition}[Branch Recovery Robustness]
\label{prop:branchRecoveryRobustness}
For all distinct $\mathbf{y},\mathbf{z}\in[\mathbf{x}]_{\mathscr{P}_I}$, we have
\begin{equation*}
    \Vert \boldsymbol{\Pi}_{\mathscr{P}_B}^T(\mathbf{y} + \mathbf{n}) - \boldsymbol{\Pi}_{\mathscr{P}_B}^T\mathbf{y} \Vert 
    < \Vert \boldsymbol{\Pi}_{\mathscr{P}_B}^T(\mathbf{y} + \mathbf{n}) - \boldsymbol{\Pi}_{\mathscr{P}_B}^T\mathbf{z} \Vert
\end{equation*}
if the disturbance is bounded by
\begin{equation*}
    \Vert \mathbf{n}\Vert < \frac{1}{2} \Vert \boldsymbol{\Pi}_{\mathscr{P}_B}^T(\mathbf{y} - \mathbf{z}) \Vert.
\end{equation*}
Moreover, this bound on the disturbance is tight.
\end{proposition}
\begin{proof}
See \cref{app:Proofs}
\end{proof}
Therefore, we would like to make $\Vert \boldsymbol{\Pi}_{\mathscr{P}_B}^T\mathbf{w} \Vert$ as large as possible for every $\mathbf{w}\in\Delta[\mathbf{x}]_{\mathscr{P}_I}$.
This will ensure that any $\mathbf{x},\mathbf{y}\in [\mathbf{x}]_{\mathscr{P}_I}$, where $\mathbf{x}\neq \mathbf{y}$, can be robustly distinguished from each other by sampling the branching coordinates $\boldsymbol{\Pi}_{\mathscr{P}_B}^T\mathbf{x}$ and $\boldsymbol{\Pi}_{\mathscr{P}_B}^T\mathbf{y}$.

Observing that $\Vert\boldsymbol{\Pi}_{\mathscr{P}_B}^T\mathbf{w}\Vert \geq \Vol{(\boldsymbol{\Pi}_{\mathscr{P}_B}^T\mathbf{w})}$, we draw a connection with the SimPQR factorization \cref{eqn:SimPQR} used to identify the immersion coordinates. 
If the total number of branches is not too large, then we can choose a collection of points $\mathbf{\bar{x}}_1,$ $\ldots,$ $\mathbf{\bar{x}}_K\in\mathcal{M}$ and perform SimPQR:
\begin{equation}
\begin{bmatrix}
\mathbf{w}_1^T \\
\vdots \\
\mathbf{w}_L^T
\end{bmatrix}
\begin{bmatrix}[c|c]
\boldsymbol{\Pi}_{\mathscr{P}_B} & \boldsymbol{\Pi}_{\mathscr{P}_B^c} \end{bmatrix} = 
\begin{bmatrix}
q_1 & & \\
& \ddots & \\
& & q_L
\end{bmatrix}
\begin{bmatrix}[c|c]
\mathbf{r}_1^{(1)} & \mathbf{r}_1^{(2)} \\
\vdots & \vdots \\
\mathbf{r}_L^{(1)} & \mathbf{r}_L^{(2)}
\end{bmatrix}, \qquad 
\left\lbrace\mathbf{w}_l \right\rbrace_{l=1}^L \triangleq \bigcup_{k=1}^K \Delta[\mathbf{\bar{x}}_k]_{\mathscr{P}_I}.
\label{eqn:SimPQR_branch_comb}
\end{equation}
This method will be preferable when there are a small number of discrete branches of the underlying manifold in the chosen set of immersion coordinates.

The number of elements in $\Delta[\mathbf{x}]_{\mathscr{P}_I}$ grows with the square of the number of branches so that performing SimPQR on $\bigcup_{k=1}^K \Delta[\mathbf{\bar{x}}_k]_{\mathscr{P}_I}$ may become intractable as the number of branches becomes large.
\Cref{prop:StrongerBranchRecovery} says that we can achieve the same branch recovery guarantee as \cref{prop:branchRecovery} (with at least as many coordinates) by finding orthonormal matrices $\lbrace \mathbf{\tilde{U}}_{\mathbf{x}}\rbrace_{\mathbf{x}\in\mathcal{M}}$ spanning the collections branch separating vectors $\Range(\mathbf{\tilde{U}}_{\mathbf{x}}) = \vspan{\left\lbrace\Delta[\mathbf{x}]_{\mathscr{P}_I}\right\rbrace}$
and choosing a set of coordinates $\mathscr{P}_B$ so that each $\boldsymbol{\Pi}_{\mathscr{P}_B}^T\mathbf{\tilde{U}}_{\mathbf{x}}$ is left-invertible.
More robust branch recovery is possible by making each $\Vol{(\boldsymbol{\Pi}_{\mathscr{P}_B}^T\mathbf{\tilde{U}}_{\mathbf{x}})}$ as large as possible.
This can be achieved by performing SimPQR on orthonormal matrices $\mathbf{\tilde{U}}_1,$ $\ldots,$ $\mathbf{\tilde{U}}_K$ spanning the branch separating vectors at the collection of points $\mathbf{\bar{x}}_1,$ $\ldots,$ $\mathbf{\bar{x}}_K\in\mathcal{M}$.
\begin{proposition}[Strong Branch Recovery]
\label{prop:StrongerBranchRecovery}
If $\Range(\mathbf{\tilde{U}}_{\mathbf{x}}) = \vspan{\left\lbrace\Delta[\mathbf{x}]_{\mathscr{P}_I}\right\rbrace}$ and \\ $\Vol{(\boldsymbol{\Pi}_{\mathscr{P}_B}^T\mathbf{\tilde{U}}_{\mathbf{x}})} > 0$, then
\begin{equation*}
    \Vert \boldsymbol{\Pi}_{\mathscr{P}_B}^T\mathbf{w}\Vert \geq
    \Vol{\left(\boldsymbol{\Pi}_{\mathscr{P}_B}^T\mathbf{\tilde{U}}_{\mathbf{x}}\right)} \Vert \mathbf{w}\Vert > 0,
    \qquad \forall \mathbf{w}\in\Delta[\mathbf{x}]_{\mathscr{P}_I},
\end{equation*}
and the condition of \Cref{prop:branchRecovery} is satisfied.
Furthermore, distinct $\mathbf{y},\mathbf{z}\in[\mathbf{x}]_{\mathscr{P}_I}$ can be robustly distinguished using the selected coordinates in the sense of \cref{prop:branchRecoveryRobustness} when the disturbance is bounded by 
\begin{equation*}
    \Vert \mathbf{n} \Vert < \frac{1}{2} \Vol{\left(\boldsymbol{\Pi}_{\mathscr{P}_B}^T\mathbf{\tilde{U}}_{\mathbf{x}}\right)} \Vert \mathbf{y} - \mathbf{z}\Vert.
\end{equation*}
\end{proposition}
\begin{proof}
See \cref{app:Proofs}.
\end{proof}
\begin{figure}
    \centering
    \subfloat[]{
    \includegraphics[width=0.4\linewidth]{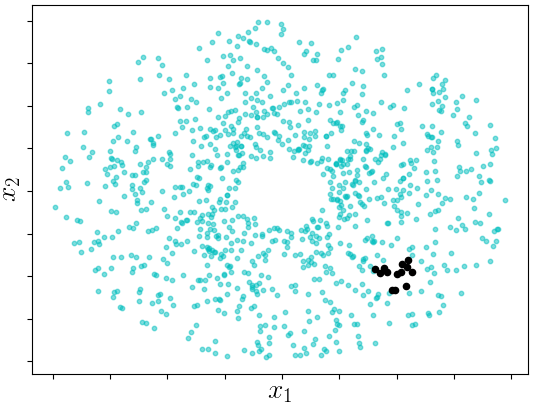}
    \label{fig:CoordinateNhd_2D}
    }
    \hfill
    \subfloat[]{
    \includegraphics[width=0.45\linewidth]{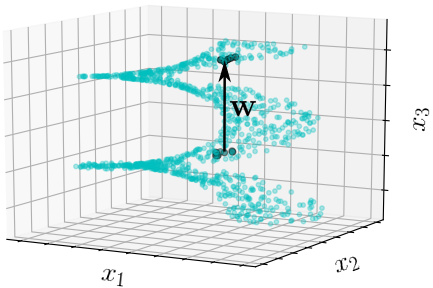}
    \label{fig:CoordinateNhd_3D}
    }
    \caption{A collection of neighboring points are found in the immersing DEIM coordinates and are observed to lie on different branches of the manifold. The branch separating vector between the cluster centers is shown. Thus, we need to add another coordinate to resolve this ambiguity.}
    \label{fig:CoordinateNhd}
\end{figure}

In practice, we can estimate $[\mathbf{\bar{x}}_k]_{\mathscr{P}_I}$ by choosing a collection of points that are close to $\mathbf{\bar{x}}_k$ in the $\mathscr{P}_I$-coordinates as in \cref{fig:CoordinateNhd_2D} and locating the corresponding cluster centroids in the full space as in \cref{fig:CoordinateNhd_3D}.
Then one may construct $\Delta[\mathbf{\bar{x}}_k]_{\mathscr{P}_I}$ by taking the differences between all the centroids in $[\mathbf{\bar{x}}_k]_{\mathscr{P}_I}$.
Alternatively, we can construct $\mathbf{\tilde{U}}_k$ simply by performing PCA on the original collection of points shown in \cref{fig:CoordinateNhd_3D}, truncating the small singular values corresponding to the width of the neighborhood.
This is much easier and effective method when there is a single branching direction as in the spiral manifold shown in \cref{fig:CoordinateNhd}.

\section{Applications}
\label{sec:applications}
\subsection{Sampling Locations in Cylinder Wake Flow}
\label{subsec:CylinderWakeSampling}

DEIM has been employed in a variety of applications to identify optimal spatial locations for sensor placement and reduced order modeling of PDEs \cite{Sargsyan2015nonlinear}.
The main idea is to find a small collection of sampling locations in physical space that can be used to reconstruct the entire solution or the nonlinear part of its time derivative.
The sparse sampling and reconstruction achieved by DEIM can be used in conjunction with a variety of methods including standard Galerkin projection and least-squares Petrov-Galerkin projection \cite{Carlberg2011efficient} to provide dimensionality reduction for modeling PDEs.
However, reduced order models employing DEIM are limited by the assumption that the solutions lie in a low-dimensional subspace that is often determined using POD to capture some large fraction of the solution's variance.
Nonlinear DEIM enables sparse sampling and reconstruction on nonlinear manifolds, opening up the possibility of further dimensionality reduction for systems evolving on such manifolds.

For example, consider the flow behind a cylinder at Reynolds number $Re = 60$ which has an unstable equilibrium shown in \cref{fig:unstableEquilibrium} and a stable limit cycle that forms the K\'{a}rm\'{a}n vortex sheet downstream of the cylinder shown in \cref{fig:stableLimitCycle}.
The same simulation data as in \cite{Otto2019linearly}, consisting of $2000$ snapshots taken at regular intervals $\Delta t = 0.2 D/U_{\infty}$, where $D$ is the cylinder diameter and $U_{\infty}$ is the free-stream flow velocity, will be used in this example.
The data was stored using the leading $200$ POD modes that capture essentially all of the variance.
Evenly numbered data points were assigned to the training data set and oddly numbered points were assigned to the testing data set.
We consider the vorticity field $\omega(\xi_1, \xi_2)$ on a down-sampled spatial grid consisting of $267$ points $-3.06 \leq \xi_1\leq 28.86$ in the streamwise direction and $67$ points $-3.96 \leq \xi_2\leq 3.96$ in the transverse direction, yielding a state vector containing the sampled vorticities $\mathbf{x}$ of dimension $n=17889$.
\begin{figure}
  \centering
  \subfloat[]{
    \includegraphics[width=0.47\textwidth]{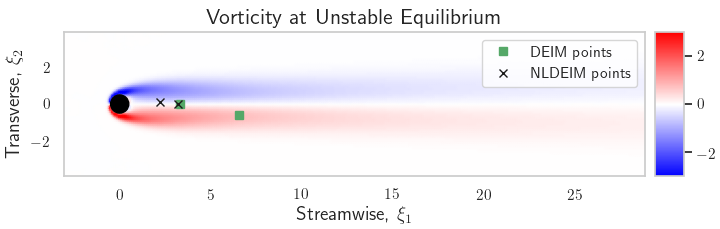}
    \label{fig:unstableEquilibrium}
  }
  \hfill
  \subfloat[]{
    \includegraphics[width=0.47\textwidth]{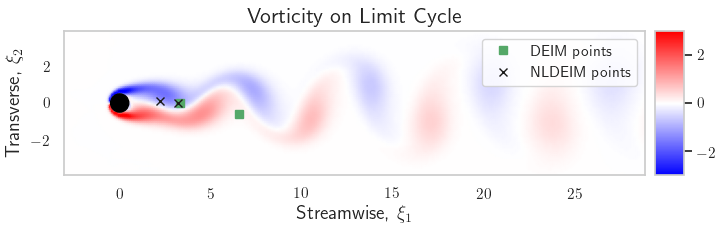}
    \label{fig:stableLimitCycle}
  }
  \label{fig:cylinderWake_exampleSnapshots}
  \caption{Cylinder wake flow snapshots at the unstable equilibrium and on the stable limit cycle. The spatial sampling locations chosen by the DEIM and nonlinear DEIM methods are also shown.}
\end{figure}

The original DEIM relies on a linear reconstruction of the state \cref{eqn:linear_DEIM_recon} whose error \cref{eqn:linear_DEIM_recon_error} scales with the amount of variance not captured by the POD subspace since
\begin{equation}
    \Expectation \Vert \boldsymbol{\Pi}_{\mathscr{P}}^T\mathbf{n} \Vert_2^2 \leq \Expectation \Vert \mathbf{n} \Vert_2^2 = \sum_{i=r+1}^{\infty} \sigma_i^2,
\end{equation}
where $\sigma_i^2$ are the principal POD variances.
Performing unstable equilibrium-subtracted POD on the training data reveals that $r=3$ POD modes are needed to capture $90\%$ of the variance and $r=9$ POD modes are needed to capture $99\%$ of the variance.
We find that a large number of POD modes must be retained in the model in order to produce accurate linear reconstructions using DEIM.
However, the state trajectories are seen to lie on a $2$-dimensional slow manifold visualized in \cref{fig:CylWake_LaplaceEigs_PODcoords} by plotting the leading graph Laplacian eigenfunctions and tangent planes over the training data in the POD coordinates.
This suggests that it may be possible to reconstruct the full state as a nonlinear function of vorticity measurements at two or more points in physical space.
\begin{figure}
  \centering
  \subfloat[]{
    \includegraphics[width=0.3\textwidth]{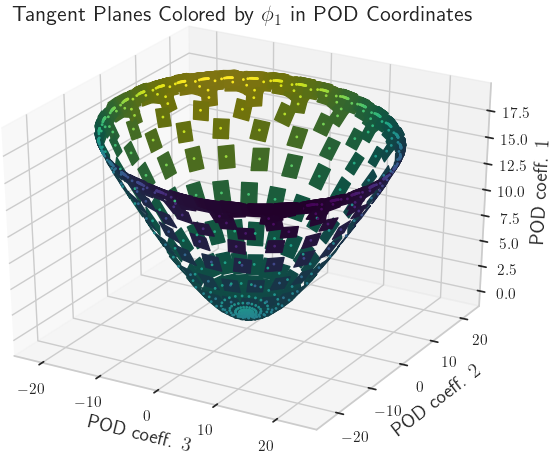}
    \label{fig:CylWake_Phi1_PODcoords}
  }
  \hfill
  \subfloat[]{
    \includegraphics[width=0.3\textwidth]{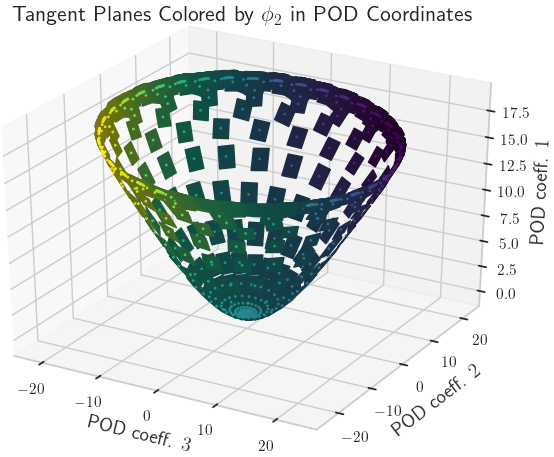}
    \label{fig:CylWake_Phi2_PODcoords}
  }
  \hfill
  \subfloat[]{
    \includegraphics[width=0.3\textwidth]{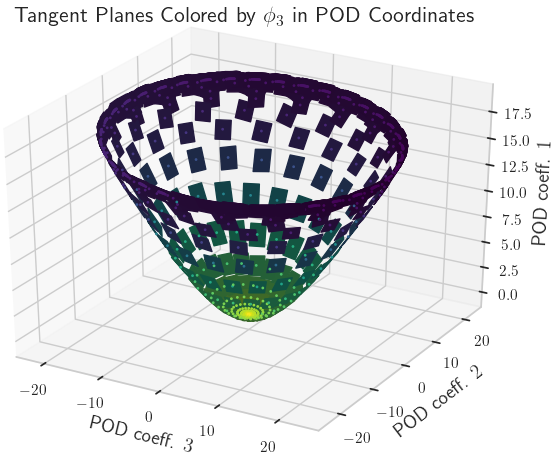}
    \label{fig:CylWake_Phi3_PODcoords}
  }
  \label{fig:CylWake_LaplaceEigs_PODcoords}
  \caption{Cylinder wake states along a trajectory plotted in the leading $3$ POD coordinates and colored according to the leading graph Laplacian eigenfunctions. Tangent planes are shown at each point. The data is observed to lie on a $2$-dimensional manifold in $17889$-dimensional space.}
\end{figure}

A naive approach that produces acceptable results in this example is to identify the sampling points using DEIM on the subspace spanned by the leading $r=2$ POD modes.
The locations in the flow identified using this method are $(\xi_1=3.30, \xi_2=0.00)$ and $(\xi_1=6.54, \xi_2=-0.60)$ which are shown in \cref{fig:cylinderWake_exampleSnapshots} using green square markers.
\Cref{fig:CylWake_LaplaceEigs_DEIMcoords} shows the state trajectories plotted in this coordinate system and colored by the Laplacian eigenfunction values.
The eigenfunction values parameterize the manifold and we see that they are single-valued functions of the DEIM coordinates since the colors in \cref{fig:CylWake_Phi1_DEIMcoords}, \cref{fig:CylWake_Phi2_DEIMcoords}, and \cref{fig:CylWake_Phi3_DEIMcoords} do not overlap.
Therefore, it is possible to reconstruct the full state $\mathbf{x}\in\mathbb{R}^{17889}$ as a nonlinear function of the vorticity measurements at the DEIM points.
\begin{figure}
  \centering
  \subfloat[]{
    \includegraphics[width=0.3\textwidth]{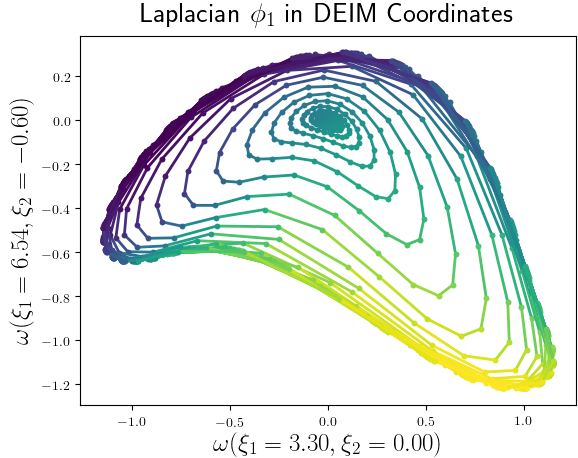}
    \label{fig:CylWake_Phi1_DEIMcoords}
  }
  \hfill
  \subfloat[]{
    \includegraphics[width=0.3\textwidth]{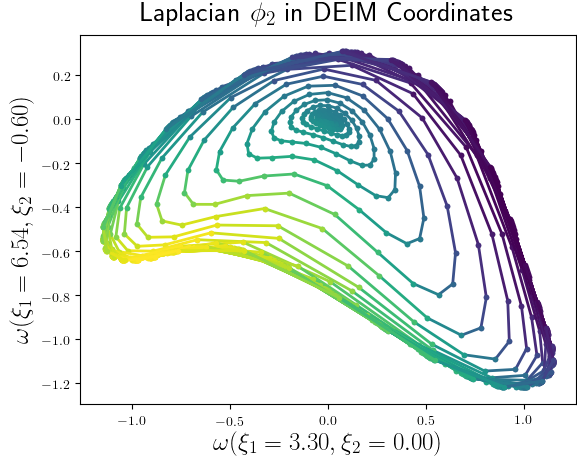}
    \label{fig:CylWake_Phi2_DEIMcoords}
  }
  \hfill
  \subfloat[]{
    \includegraphics[width=0.3\textwidth]{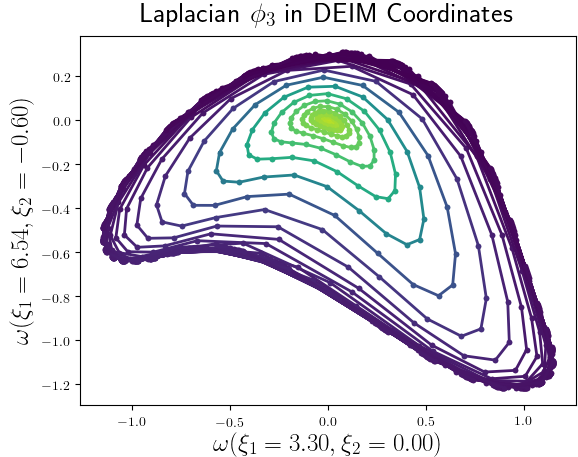}
    \label{fig:CylWake_Phi3_DEIMcoords}
  }
  \label{fig:CylWake_LaplaceEigs_DEIMcoords}
  \caption{Cylinder wake state trajectories plotted in the coordinates identified by DEIM and colored according to the leading graph Laplacian eigenfunctions. 
  Because the Laplacian eigenfunctions are single-valued in this coordinate system (colors do not overlap), it is possible to (nonlinearly) reconstruct the state from these measurements.}
\end{figure}

However, the POD-DEIM approach described above is not guaranteed to identify points that can be used to reconstruct the full state as a nonlinear function, as will be shown using a concrete example in \cref{subsec:Burgers}.
On the other hand, the nonlinear DEIM presented in this paper does provide this guarantee by construction.
In order to apply NLDEIM to the cylinder wake data, we found tangent planes to the underlying manifold at each training data point using the Laplacian eigenmaps-based algorithm discussed in \cref{app:LapEigAndTanSpace}.
SimPQR with $\gamma = 1/K = 10^{-3}$ and $\epsilon = 10^{-4}$ was then used to find immersing NLDEIM coordinates for immersing the manifold.
The locations in the flow identified using this method are $(\xi_1=3.18, \xi_2=0.00)$ and $(\xi_1=2.22, \xi_2=0.12)$ which are shown in \cref{fig:cylinderWake_exampleSnapshots} using black ``x'' markers.
Since there were no global ambiguities between points on the manifold in these coordinates, the immersing NLDEIM coordinates proved to be sufficient to embed the manifold and hence to reconstruct the entire state.
The state trajectories in the NLDEIM coordinates are shown in \cref{fig:CylWake_LaplaceEigs_NLDEIMcoords}, which is a graphical demonstration that the full state can be reconstructed using only these measurements.
\begin{figure}
  \centering
  \subfloat[]{
    \includegraphics[width=0.3\textwidth]{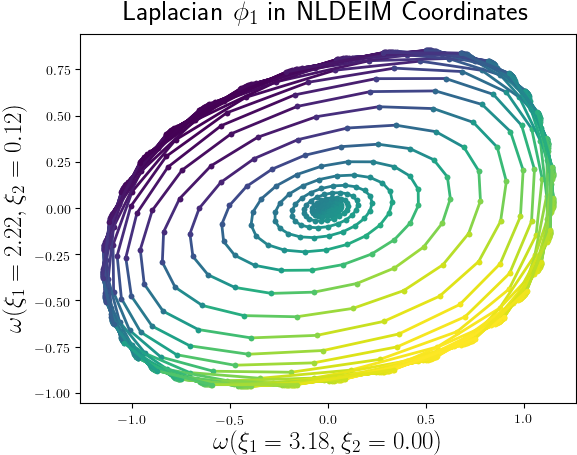}
    \label{fig:CylWake_Phi1_NLDEIMcoords}
  }
  \hfill
  \subfloat[]{
    \includegraphics[width=0.3\textwidth]{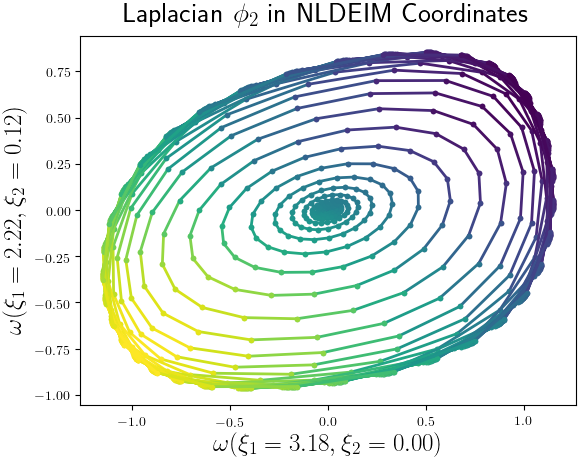}
    \label{fig:CylWake_Phi2_NLDEIMcoords}
  }
  \hfill
  \subfloat[]{
    \includegraphics[width=0.3\textwidth]{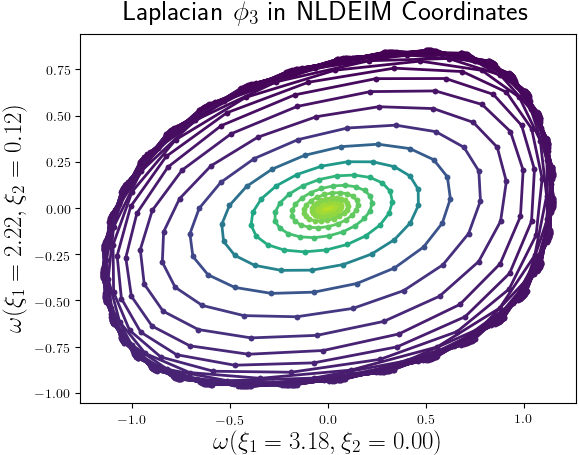}
    \label{fig:CylWake_Phi3_NLDEIMcoords}
  }
  \label{fig:CylWake_LaplaceEigs_NLDEIMcoords}
  \caption{Cylinder wake state trajectories plotted in the coordinates identified by NLDEIM and colored according to the leading graph Laplacian eigenfunctions. 
  Because the Laplacian eigenfunctions are single-valued in this coordinate system (colors do not overlap), it is possible to (nonlinearly) reconstruct the state from these measurements.}
\end{figure}

How does NLDEIM compare to the original approach using POD and DEIM on this example?
One way to make this comparison is to quantify how robustly the full state can be reconstructed using the respective coordinates in the presence of small noise.
The locally linear reconstruction \cref{eqn:locallyLinearRecon0} amplifies noise in the selected coordinates through the matrix $\left(\boldsymbol{\Pi}_{\mathscr{P}}^T\mathbf{U}_{\mathbf{\bar{x}}}\right)^{+}$, whose largest singular value determines the worst-case amplification on that patch.
The distribution of noise amplification levels over patches centered at each training data point are plotted in \cref{fig:CylWake_amplification}, showing that the NLDEIM coordinates enable significantly more robust reconstructions than the coordinates chosen using POD and DEIM.
In fact, the boxplot in \cref{fig:CylWake_amplification} shows that the median amplification using the DEIM coordinates is higher than the maximum amplification using the NLDEIM coordinates.
\begin{figure}
    \centering
    \includegraphics[width=0.5\linewidth]{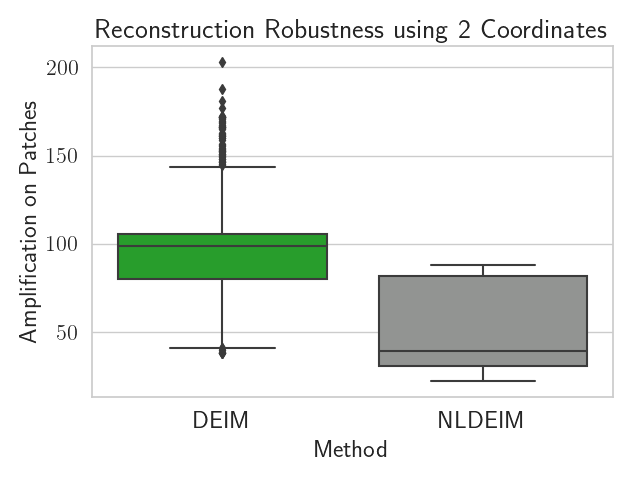}
    \caption{Distribution of largest singular value of $\left(\boldsymbol{\Pi}_{\mathscr{P}}^T\mathbf{U}_{\mathbf{\bar{x}}}\right)^{+}$ over patches at every training data point. We see that the coordinates chosen by NLDEIM amplify noise less than those chosen by DEIM.}
    \label{fig:CylWake_amplification}
\end{figure}

\subsection{POD Mode Selection in Cylinder Wake Flow}
\label{subsec:CylinderWakePOD}
It is also possible to use the Nonlinear DEIM technique to identify a collection of POD modes to which the state of a system is enslaved.
Common wisdom says that the state will be described by the most energetic modes, yet this is not true in general unless those modes capture essentially all of the variance of the state.
In applications where we want to build very low-order models, we cannot include all of the energetic modes, so it is essential to choose the smallest collection that are needed to fully describe the state.

Let us consider the problem of choosing between the leading $200$ POD modes used to describe the same cylinder wake flow as in \cref{subsec:CylinderWakeSampling}.
The state trajectories in the leading two POD modes shown in \cref{fig:cylinderWake_modes1and2} are plotted in \cref{fig:CylWake_LaplaceEigs_2PODcoords} and are colored according to the leading three eigenfunctions of the graph Laplacian.
We clearly see that the Laplacian eigenfunction values are not single-valued functions of the leading two POD coordinates. Furthermore, there are even points where the trajectory is perpendicular to these coordinates.
Therefore, the slow manifold in the cylinder wake flow cannot be immersed or embedded using the coefficients on these modes and attempts to build a reduced order model based on them would be doomed.
\begin{figure}
  \centering
  \subfloat[]{
    \includegraphics[width=0.47\textwidth]{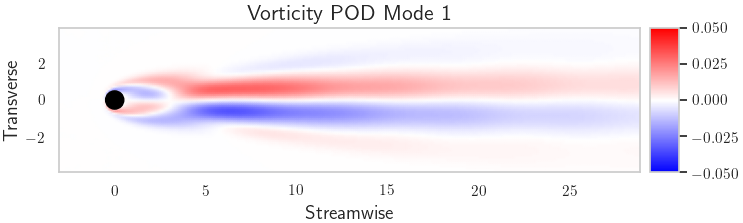}
  }
  \hfill
  \subfloat[]{
    \includegraphics[width=0.47\textwidth]{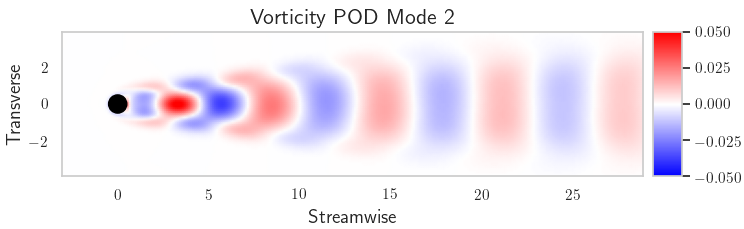}
  }
  \label{fig:cylinderWake_modes1and2}
  \caption{The leading two equilibrium-subtracted POD modes for the cylinder wake flow computed using the training data.}
\end{figure}
\begin{figure}
  \centering
  \subfloat[]{
    \includegraphics[width=0.3\textwidth]{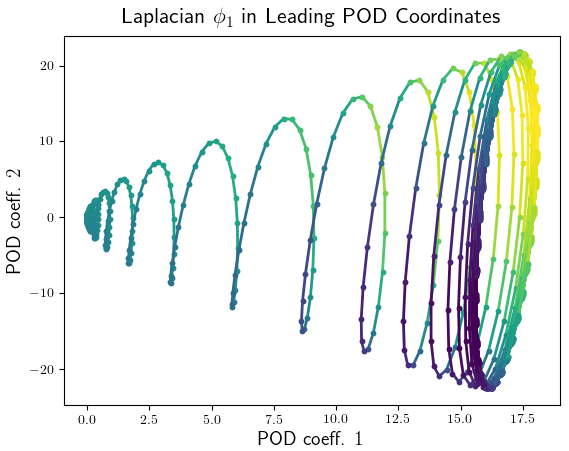}
    \label{fig:CylWake_Phi1_2PODcoords}
  }
  \hfill
  \subfloat[]{
    \includegraphics[width=0.3\textwidth]{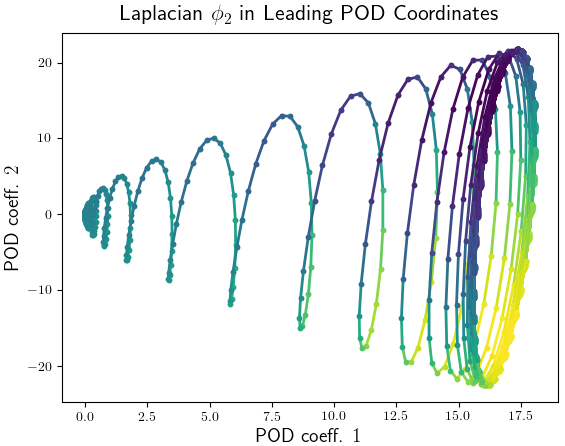}
    \label{fig:CylWake_Phi2_2PODcoords}
  }
  \hfill
  \subfloat[]{
    \includegraphics[width=0.3\textwidth]{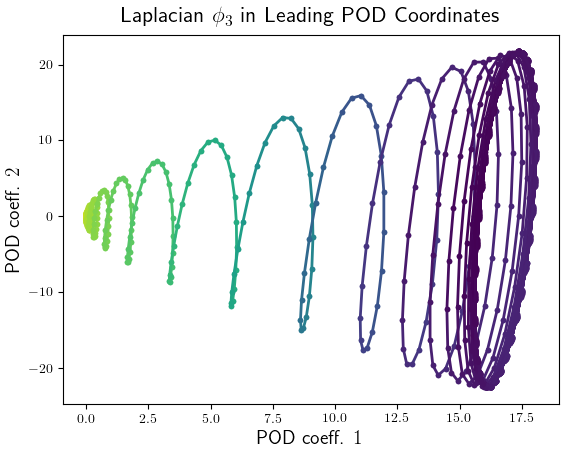}
    \label{fig:CylWake_Phi3_2PODcoords}
  }
  \label{fig:CylWake_LaplaceEigs_2PODcoords}
  \caption{Cylinder wake state trajectories plotted in the leading two POD coordinates and colored according to the leading graph Laplacian eigenfunctions. 
  Because the Laplacian eigenfunctions are not single-valued in this coordinate system, it is impossible to reconstruct the state from these measurements.}
\end{figure}

On the other hand, one easily sees that the coefficients on POD modes $2$ and $3$ shown in \cref{fig:cylinderWake_modes2and3} can be used to embed the underling manifold.
Indeed, these modes are recovered by performing NLDEIM using SimPQR with $\gamma = 1/K = 10^{-3}$ and $\epsilon = 10^{-4}$ on the tangent planes to the training data in the leading $n=200$ POD coordinates.
One can see in \cref{fig:CylWake_LaplaceEigs_PODcoords} that the tangent planes all have non-singular projections into the $\mathscr{P}=\lbrace 2,3 \rbrace$ coordinate plane.
The trajectory of the cylinder wake flow is plotted in these coordinates in \cref{fig:CylWake_LaplaceEigs_2NLDEIMPODcoords}, clearly illustrating that each point can be reconstructed using the POD coefficients selected by NLDEIM.
These coordinates are well-suited to building reduced-order models, either based on nonlinear Galerkin or purely data-driven methods.
\begin{figure}
  \centering
  \subfloat[]{
    \includegraphics[width=0.47\textwidth]{Figures/CylinderWake_POD/CylWake_Vorticity_PODmode2.PNG}
  }
  \hfill
  \subfloat[]{
    \includegraphics[width=0.47\textwidth]{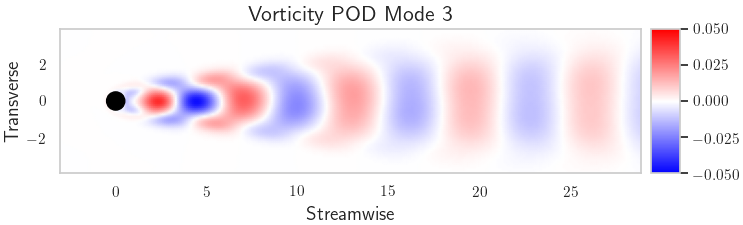}
  }
  \label{fig:cylinderWake_modes2and3}
  \caption{Equilibrium-subtracted POD modes $2$ and $3$ for the cylinder wake flow computed using the training data.}
\end{figure}
\begin{figure}
  \centering
  \subfloat[]{
    \includegraphics[width=0.3\textwidth]{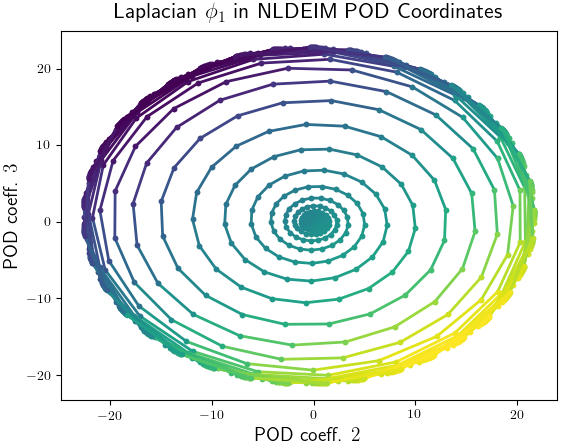}
    \label{fig:CylWake_Phi1_2NLDEIMPODcoords}
  }
  \hfill
  \subfloat[]{
    \includegraphics[width=0.3\textwidth]{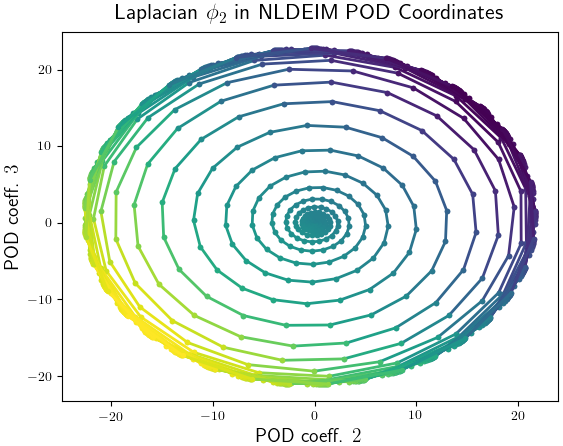}
    \label{fig:CylWake_Phi2_2NLDEIMPODcoords}
  }
  \hfill
  \subfloat[]{
    \includegraphics[width=0.3\textwidth]{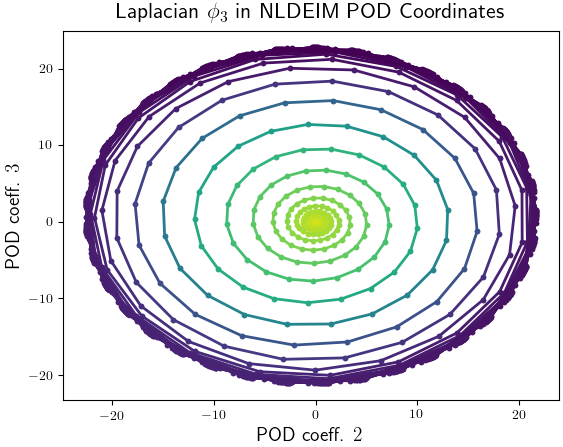}
    \label{fig:CylWake_Phi3_2NLDEIMPODcoords}
  }
  \label{fig:CylWake_LaplaceEigs_2NLDEIMPODcoords}
  \caption{Cylinder wake state trajectories plotted in the POD coordinates $\mathscr{P}=\lbrace{2,3}$ selected by NLDEIM and colored according to the leading graph Laplacian eigenfunctions. 
  Because the Laplacian eigenfunctions are single-valued in this coordinate system, it is possible to reconstruct the state from these measurements.}
\end{figure}

\subsection{Sampling Locations for Burgers Equation}
\label{subsec:Burgers}
In this section, we present an example where choosing a small set of coordinates using POD and DEIM on a subspace that does not capture all of the data's variance fails to embed the manifold on which the data lies and therefore does not allow for reconstruction of the full state.
In contrast, NLDEIM using SimPQR succeeds in identifying suitable embedding coordinates that are empirically and intuitively justified, yielding further insight about the problem.
Consider the viscous Burgers equation with periodic boundary conditions,
\begin{equation}
    u_t + u u_{\xi} = \nu u_{\xi\xi}, \quad 0\leq \xi\leq 1, \quad 0\leq t\leq 1, \quad u(\xi+1,t) = u(\xi,t),
\end{equation}
starting from triangularly-shaped initial conditions
\begin{equation}
    u(\xi,0) = \max{\left\lbrace 0, 1 - \frac{2}{\chi}\vert \xi - 0.5 \vert \right\rbrace},
\end{equation}
having different widths $0.1\leq \chi\leq 0.5$.
If we discretize the PDE using $n=256$ spatial grid points then the snapshots $\mathbf{u}$ of the solution lie on a $2$-dimensional manifold $\mathcal{M}\subset\mathbb{R}^{256}$ parameterized by the time $t$ and width $\chi$ of the initial condition.
Some snapshots from simulations ranging over $0\leq t\leq 1$ at the two extreme initial conditions and $\nu = 10^{-3}$ are shown in \cref{fig:BurgersSimPlts}.
Simulations were performed at $25$ evenly spaced widths ranging over $0.1\leq \chi\leq 0.5$.
$500$ snapshots were recorded from each simulation every $\Delta t = 0.002$.
All the snapshots across the different simulations were used as the training data in this example.
\begin{figure}
    \centering
    \subfloat[]{
    \includegraphics[width=0.45\linewidth]{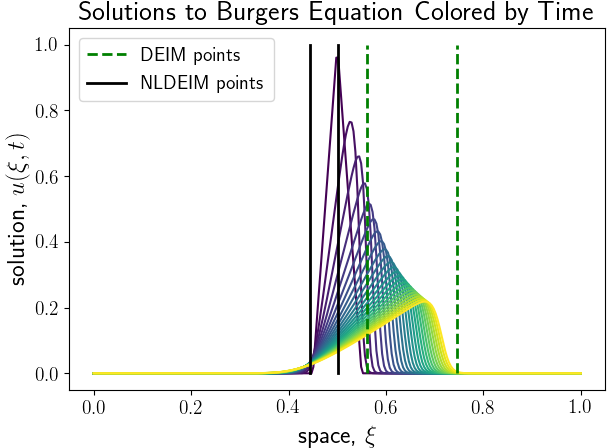}
    \label{fig:BurgersSimPlt01}
    }
    \hfill
    \subfloat[]{
    \includegraphics[width=0.45\linewidth]{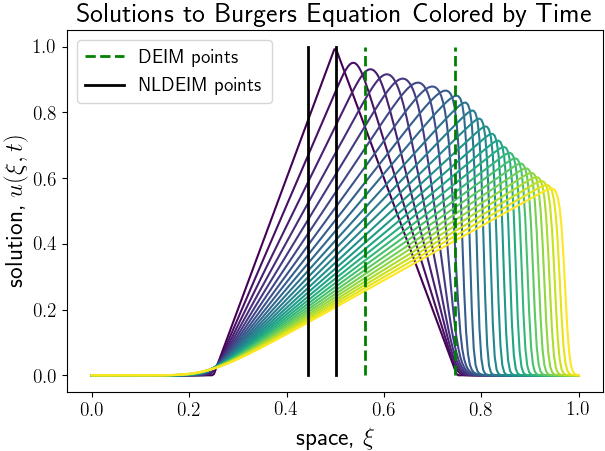}
    \label{fig:BurgersSimPlt05}
    }
    \caption{Example simulations of the Burgers equation from the two extreme initial conditions. A collection of snapshots is plotted for each initial condition and colored according to the simulation time. The locations of the points selected by DEIM and Nonlinear DEIM are indicated by the dashed green and solid black vertical lines respectively.}
    \label{fig:BurgersSimPlts}
\end{figure}

It is easy to see that even though the manifold $\mathcal{M}$ of solution snapshots is $2$-dimensional, it is curved and cannot be captured by any low-dimensional subspace.
Equivalently, the solution snapshots in \cref{fig:BurgersSimPlt01} and \cref{fig:BurgersSimPlt05} and all those cases in between cannot be accurately represented through superpositions of a small number of modes because the solutions have peaks that translate in space.
In fact, $15$ mean-subtracted POD modes are needed to capture $99\%$ of the training data's variance and the remaining principal variances continue to decay slowly.
Consequently, we cannot expect that performing DEIM using the leading $2$ POD modes, yielding measurement locations $\xi = 0.73$ and $\xi = 0.58$ shown as vertical dashed green lines in \cref{fig:BurgersSimPlts}, will reveal coordinates that immerse or embed the manifold on which the data lies.
Indeed, plotting the training data in the DEIM coordinates and coloring each point by the parameterizing coordinates time, $t$ and initial width, $\chi$ in \cref{fig:Burgers_TimeWidth_DEIMcoords}, we see that the manifold cannot be reconstructed from the coordinates selected by DEIM.
\begin{figure}
  \centering
  \subfloat[]{
    \includegraphics[width=0.45\textwidth]{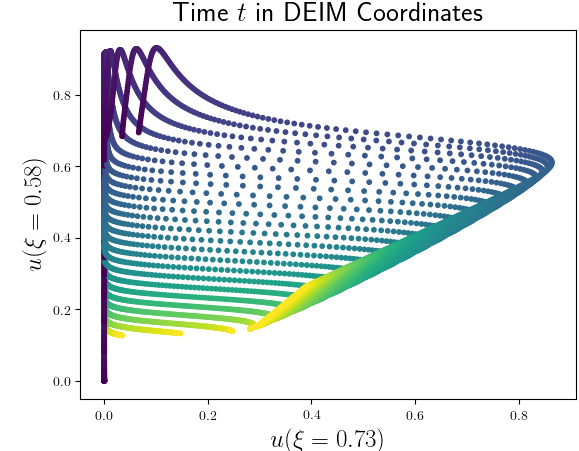}
    \label{fig:Burgers_Time_DEIMcoords}
  }
  \hfill
  \subfloat[]{
    \includegraphics[width=0.45\textwidth]{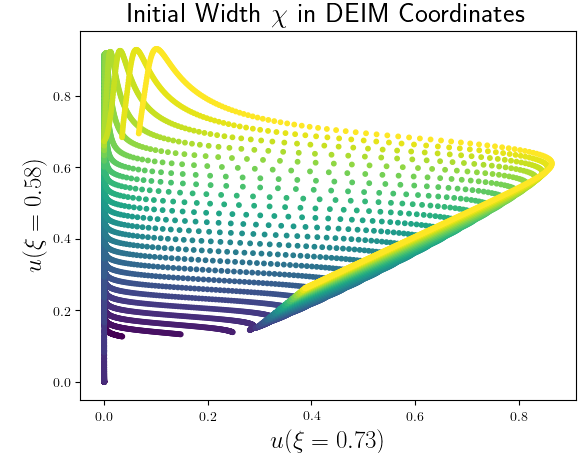}
    \label{fig:Burgers_Width_DEIMcoords}
  }
  \label{fig:Burgers_TimeWidth_DEIMcoords}
  \caption{Burgers equation snapshots plotted in the coordinates identified by DEIM and colored according to the parameterizing coordinates time, $t$ and initial width, $\chi$. 
  We see that the underlying manifold folds back on itself when projected into the DEIM coordinates; hence the parameterizing coordinates are not single-valued functions of the coordinates selected by DEIM.}
\end{figure}

As discussed briefly in \cref{rem:TangentSpacesFromData}, our training data in this example comes from experiments performed over a regular grid of parameters $t$ and $\chi$.
Therefore, the tangent vectors can be approximated simply by orthonormalizing the finite differences
\begin{equation}
\begin{bmatrix}
\mathbf{v}_{t}(t_{i},\chi_{j}) & \mathbf{v}_{\chi}(t_{i},\chi_{j})
\end{bmatrix} = \mathbf{U}_{i,j}\mathbf{R}_{i,j}, \qquad
\left\lbrace
    \begin{aligned}
    \mathbf{v}_{t}(t_{i},\chi_{j}) &= \mathbf{u}(t_{i+1},\chi_{j}) - \mathbf{u}(t_{i},\chi_{j}) \\
    \mathbf{v}_{\chi}(t_{i},\chi_{j}) &= \mathbf{u}(t_{i},\chi_{j + 1}) - \mathbf{u}(t_{i},\chi_{j})
    \end{aligned}
\right.,
\end{equation}
at the training points $\mathbf{u}(t_{i},\chi_{j})$ for $1\leq i < 500$ and $1\leq j < 25$.
The tangent planes at $K=5000$ randomly selected training data points were used to perform SimPQR with $\epsilon = 10^{-6}$ and $\gamma = 1/K = 2\times 10^{-4}$ to reveal the immersing Nonlinear DEIM sampling locations $\xi = 0.44$ and $\xi = 0.50$ shown as vertical black lines in \cref{fig:BurgersSimPlts}.
In this case, the immersing NLDEIM coordinates are sufficient to embed the underlying manifold and reconstruct the full state without any additional branching coordinates.
This is illustrated in \cref{fig:Burgers_TimeWidth_NLDEIMcoords} by plotting the training snapshots in the NLDEIM coordinate system and coloring the points according to the parameterizing coordinates $t$ and $\chi$.
It is clear from \cref{fig:Burgers_TimeWidth_NLDEIMcoords} that the parameterizing coordinates and hence the entire state are single-valued functions of the selected NLDEIM coordinates.
We observe that one NLDEIM sampling point marks the center of the domain and the peak of each initial condition.
The other NLDEIM sampling point marks the left-most edge where the narrowest initial condition is supported.
Intuitively from \cref{fig:BurgersSimPlts}, we see that $u(\xi=0.44,t)$ and $u(\xi=0.50,t)$ can be used to reconstruct the initial width of the blip $\chi$ as well as the time of the snapshot $t$.
\begin{figure}
  \centering
  \subfloat[]{
    \includegraphics[width=0.45\textwidth]{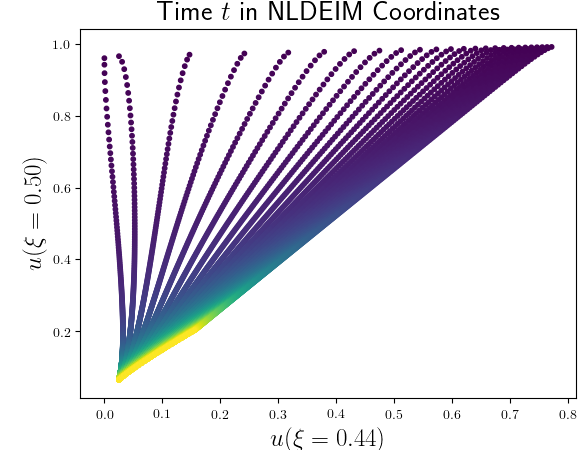}
    \label{fig:Burgers_Time_NLDEIMcoords}
  }
  \hfill
  \subfloat[]{
    \includegraphics[width=0.45\textwidth]{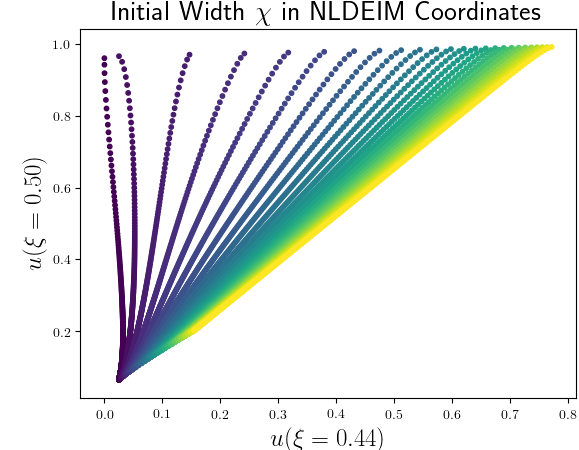}
    \label{fig:Burgers_Width_NLDEIMcoords}
  }
  \label{fig:Burgers_TimeWidth_NLDEIMcoords}
  \caption{Burgers equation snapshots plotted in the coordinates identified by NLDEIM and colored according to the parameterizing coordinates time, $t$ and initial width, $\chi$. 
  We see that the parameterizing coordinates and hence the entire state are single-valued functions of the NLDEIM coordinates.}
\end{figure}

To further illustrate the difference between the coordinates chosen using DEIM and NLDEIM, we trained simple distance-weighted $3$-nearest neighbor regression models to reconstruct the remaining state variables using scikit-learn.
The testing data set consisted of $25$ simulations at randomly chosen $\chi\in[0.1,\ 0.5]$ from which $50$ snapshots were collected, totaling $1250$ examples.
The empirical distributions of the relative error $\Vert \mathbf{\hat{u}} - \mathbf{u} \Vert_2 /\Vert \mathbf{u} \Vert_2$ are plotted in \cref{fig:BurgersDataErrors}.
Indeed, the small relative errors using the NLDEIM coordinates indicate that they contain all or nearly all information about the state $\mathbf{u}\in\mathcal{M}$.
Different regression or interpolation methods might yield even lower reconstruction errors using the NLDEIM coordinates.
However, any reconstruction using the DEIM coordinates is limited by the fact that there are local and global ambiguities between points on the underlying manifold in these coordinates as seen in \cref{fig:Burgers_TimeWidth_DEIMcoords}, hence the state is not actually a function of these coordinates.
\begin{figure}
    \centering
    \includegraphics[width=0.5\linewidth]{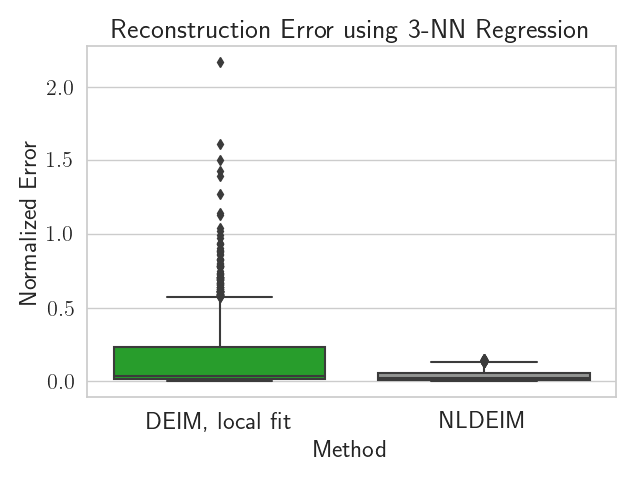}
    \caption{Empirical relative error distributions on the testing data set when using DEIM and NLDEIM coordinates to predict the remaining state variables with $3$-nearest neighbors regression. The $l_2$ error is normalized by the solution magnitude.}
    \label{fig:BurgersDataErrors}
\end{figure}

\section{Conclusion and Future Work}
In this paper we have explored how to find small collections of immersing and embedding coordinates for data lying on low-dimensional manifolds in high-dimensional space.
By extending the well-understood case of a linear manifold that leads to the DEIM algorithm, we formulated a novel nonlinear DEIM algorithm for identifying coordinates that parameterize all tangent spaces to a nonlinear manifold.
This lead to a trade-off between choosing fewer coordinates by sharing pivot columns across more patches and tangent space reconstruction robustness which is handled through a user-defined parameter in the SimPQR algorithm.
We find that in different parameter regimes, we can achieve guarantees regarding either the robustness or parsimony of coordinate choices at each stage of the algorithm.
At one extreme, SimPQR chooses the coordinates that can be applied across the largest number of patches and at the other extreme SimPQR chooses columns that are Businger-Golub pivots on every patch to which they are applied.
Furthermore, it is easy to construct every possible solution in between these extremes since at any given value of the parameter, we can find the next value of the parameter that produces a change in the solution.

The coordinates discovered by applying our SimPQR algorithm to a collection of tangent space bases eliminated continuous ambiguities between points, but did not resolve discrete ambiguities arising from the global manifold geometry.
We proposed a method for resolving this issue by applying the same SimPQR algorithm to a different collection of vectors or matrices describing the gaps between different branches of the underlying manifold in the immersion coordinates.

In addition to toy examples, we considered two real-world examples coming from high-dimensional dynamical systems.
In the cylinder wake example both DEIM and nonlinear DEIM recovered sets of coordinates that embed the underlying slow manifold, however, the nonlinear DEIM embedding was shown to be superior from the perspective of tangent space reconstruction robustness.
Of course DEIM is not guaranteed to find a set of coordinates that immerse or embed the manifold unless the subspace on which DEIM is performed captures essentially all of the data's variance.
The Burgers equation example provides an example where the data lies on a two-dimensional manifold that is not captured in any low-dimensional subspace.
When DEIM is applied to the leading POD modes, it yields a set of coordinates which fail to immerse or embed the manifold.
On the other hand nonlinear DEIM is guaranteed by construction to produce an immersion.
In fact the coordinates identified using nonlinear DEIM are not only provide an immersion, but also an embedding in this example.
Furthermore, the sampling locations revealed by nonlinear DEIM on the Burgers equation correspond to our physical intuition about the problem.
The SimPQR method developed in this paper for nonlinear DEIM proved to be an efficient method that successfully recovers nearly optimal collections of immersion and embedding coordinates on toy and real-world examples.

There are many avenues for future work related to nonlinear DEIM described in the following short subsections.

\subsection{Model Reduction via Manifold Projection}
Recent model reduction techniques require performing of a Galerkin-like projection of the nonlinear dynamics onto a collection of locally-optimal subspaces \cite{Amsallem2012nonlinear, Peng2016nonlinear} or onto a learned solution manifold \cite{Lee2018model}.
Nonlinear DEIM allows us to identify a global set of reduced coordinates that may be used to accelerate the projection operations required by the manifold Petrov-Galerkin procedures of \cite{Lee2018model} during the online phase.

However, one part of the problem that we did not address in this paper is how to reconstruct the state from the sampled coordinates;
this will require implementing an efficient and accurate interpolation or regression algorithm.
Some possibilities include Gaussian process regression \cite{Rasmussen2003gaussian} or Fefferman's interpolation method \cite{Fefferman2005interpolation, Fefferman2009whitney}.

\subsection{Curvature-Based Patch Selection}
Throughout this paper we have assumed that a representative collection of patches has been found that provide an accurate locally linear approximation of the underlying manifold.
In our examples, we centered the patches either at the training data points or at a randomly chosen subset of the training data.
However, since the proposed SimPQR algorithm scales with $\mathcal{O} (K\log K)$ where $K$ is the number of patches, it is desirable to choose as few patches as necessary to approximate the manifold.
A higher concentration of patches will be needed in regions where the manifold is highly curved whereas a lower concentration is needed in regions where the manifold is approximately flat.
Therefore, a future avenue could involve choosing patch locations based on local manifold curvature to achieve a certain level of accuracy for the locally linear approximation.

\subsection{Noise Robustness and Alternative Formulations}
The work reported in this paper has primarily focused on recovering coordinates that robustly immerse or embed the underlying manifolds near which data lie.
The notions of robustness considered here were minimal noise amplification for immersing coordinates and maximum distance between branches in the branching coordinates.
Of course there are many other notions of robustness that we intend to consider in future work.
In particular, we may have some information about the probability distribution of the noise, allowing us to formulate a greedy algorthm for NLDEIM that over-samples the state in order to reject the noise.

It may also be possible to implement a ``strong'' version of SimPQR by performing additional pivoting operations based on the strong PQR factorizaton method of \cite{Gu1996}.
Such a method may be capable of finding better collections of nonlinear DEIM coordinates as measured by the minimum patch volume.

\subsection{More Efficient Algorithms}
The SimPQR method developed in this paper relies on storing and accessing the entire tangent space matrices $\lbrace \mathbf{U}_k \rbrace_{k=1}^K$ which will become a memory bottle-neck when the state dimension is large.
For example, we may have millions of states when applying nonlinear DEIM to data sets coming from fluid dynamics simulations.
In order to avoid working with the full tangent space matrices, we may implement a simultaneous variant of the so-called ``communication-avoiding'' PQR algorithm \cite{Demmel2015}.
Such an algorithm would employ ``tournament pivoting'' by finding nonlinear DEIM coordinates immersing hierarchical sub-collections of states.
The nonlinear DEIM coordinates on smaller sub-collections will be used to find the nonlinear DEIM coordinates on larger sub-collections and so on until a collection is found that immerses the entire state space manifold.


\section*{Acknowledgments}
We would like to thank Scott Dawson for providing us with the data from his
cylinder wake simulations. 
We would also like to thank Andres Goza, Michael Mueller, and Alberto Padovan for helpful discussions about nonlinear DEIM and its applications.
The authors gratefully acknowledge funding from the Army Research Office, under award number W911NF-17-1-0512.

\bibliographystyle{siamplain}
\bibliography{RowleyGroupReferences}

\appendix

\section{Householder Transformations}
\label{app:Householder}
Once the pivot column $j^*$ is chosen during the $i$th stage of PQR \cref{alg:PQR}, the Householder vector and coefficient $(\mathbf{v}_i,\ \beta) = \house(\mathbf{A}(i:r,j^*))$ are computed using \cref{alg:House}.
\begin{algorithm}
  \caption{Householder Vector, Algorithm 5.1.1 in \cite{Golub2013matrix}}
  \label{alg:House}
  \textbf{function} $(\mathbf{v},\ \beta) = \house{(\mathbf{x})}$
  \begin{algorithmic}
    \STATE{$m=\length{(\mathbf{x})}$, $\sigma = \mathbf{x}(2:m)^T\mathbf{x}(2:m)$}
    \STATE{initialize $\mathbf{v} = \begin{bmatrix}
    1 \\ \mathbf{x}(2:m)
    \end{bmatrix}$}
    \COMMENT{Householder vector}
    \IF[$\mathbf{x}$ already aligned with $\mathbf{e}_1$]{$\sigma = 0$ and $\mathbf{x}(1) \geq 0$}
        \STATE{$\beta = 0$}
        \COMMENT{Householder reflection is trivial}
    \ELSIF[$\mathbf{x}$ already aligned with $-\mathbf{e}_1$]{$\sigma = 0$ and $\mathbf{x}(1) < 0$}
        \STATE{$\beta = -2$}
        \COMMENT{Householder reflection flips sign of $\mathbf{x}$}
    \ELSE[construct transformation in non-trivial case]
        \STATE{$\mu = \sqrt{\mathbf{x}(1)^2 + \sigma}$}
        \IF{$\mathbf{x}(1) \leq 0$}
            \STATE{$\mathbf{v}(1) = \mathbf{x}(1) - \mu$}
        \ELSE
            \STATE{$\mathbf{v}(1) = -\sigma/(\mathbf{x}(1) + \mu)$}
        \ENDIF
        \STATE{$\beta = 2\mathbf{v}(1)^2/(\sigma + \mathbf{v}(1)^2)$}
        \STATE{$\mathbf{v} \leftarrow \mathbf{v}/\mathbf{v}(1)$}
    \ENDIF
    
    \RETURN{$\mathbf{v}$, $\beta$}
    \COMMENT{householder vector and coefficient}
  \end{algorithmic}
\end{algorithm}
The Householder reflection
\begin{equation}
    \mathbf{H}_i = \mathbf{I}_r - \underbrace{\frac{2}{\Vert \mathbf{v}_i \Vert^2}}_{\beta}\mathbf{\tilde{v}}_i\mathbf{\tilde{v}}_i^T, \qquad 
    \mathbf{\tilde{v}}_i = \begin{bmatrix}
    \mathbf{0}_{i-1} \\
    \mathbf{v}_i
    \end{bmatrix},
\end{equation}
is a unitary transformation that zeros out rows $i+1$ through $r$ when applied to the pivot column:
\begin{equation}
    \mathbf{H}_i\begin{bmatrix}
    \mathbf{A}(1:i-1,j^*) \\
    \mathbf{A}(i,j^*) \\
    \mathbf{A}(i+1:r,j^*)
    \end{bmatrix} = 
    \begin{bmatrix}
    \mathbf{A}(1:i-1,j^*) \\
    \Vert \mathbf{A}(i:r,j^*) \Vert \\
    \mathbf{0}_{r-i}
    \end{bmatrix}.
\end{equation}
The Householder vectors are all scaled so that $\mathbf{v}_i(1) = 1$ and the remaining entries in the vectors can be stored by over-writing the zeroed out entries in the pivot column so that we obtain
\begin{multline}
    \mathbf{A}^T\begin{bmatrix}[c|c]
    \boldsymbol{\Pi}_{\mathscr{P}} & \boldsymbol{\Pi}_{\mathscr{P}^c} 
    \end{bmatrix} = \\
    \begin{bmatrix}[ccccc|ccc]
    \mathbf{R}(1,1) & \mathbf{R}(1,2) & \cdots & \mathbf{R}(1,r-1) & \mathbf{R}(1,r) & \mathbf{R}(1,1) & \cdots & \mathbf{R}(1,n-r) \\
    \mathbf{v}_{1}(2) & \mathbf{R}(2,2) & \cdots & \mathbf{R}(2,r-1) & \mathbf{R}(2,r) & \mathbf{R}(2,1) & \cdots & \mathbf{R}(2,n-r) \\
    \mathbf{v}_{1}(3) & \mathbf{v}_{2}(2) & \cdots & \mathbf{R}(3,r-1) & \mathbf{R}(3,r) & \mathbf{R}(3,1) & \cdots & \mathbf{R}(3,n-r) \\
    \vdots & \vdots & \ddots & \vdots & \vdots & \vdots & \ddots & \vdots \\
    \mathbf{v}_{1}(r) & \mathbf{v}_{2}(r-1) & \cdots & \mathbf{v}_{r-1}(2) & \mathbf{R}(r,r) & \mathbf{R}(r,1) & \cdots & \mathbf{R}(r,n-r) \\
    \end{bmatrix}
\end{multline}
as the final result of \cref{alg:PQR}.
As described in \cite{Golub2013matrix}, the matrix
\begin{equation}
    \mathbf{Q} = \mathbf{H}_1 \mathbf{H}_2 \cdots \mathbf{H}_{r-1}
\end{equation}
can be formed by successive application of the Householder transformations.

The same construction can be used to find the matrices $\mathbf{Q}_k$ and $\mathbf{\tilde{R}}_k$ using the patch-wise ordered lists of pivot columns $\mathscr{P}_k$, the remaining columns $\mathscr{P}_k^c$ and the matrices $\mathbf{A}_k$ returned by SimPQR \cref{alg:SimPQR}.
Forming permutation matrices $\boldsymbol{\Pi}_{\mathscr{P}_k}$ and $\boldsymbol{\Pi}_{\mathscr{P}_k^c}$ from the index lists $\mathscr{P}_k$ and $\mathscr{P}_k^c$ respectively, the entries in the returned matrices
\begin{multline}
    \mathbf{A}_k^T \begin{bmatrix}[c|c]
    \boldsymbol{\Pi}_{\mathscr{P}_k} & \boldsymbol{\Pi}_{\mathscr{P}_k^c}
    \end{bmatrix} = \\ 
    \begin{bmatrix}[ccccc|ccc]
    \mathbf{\tilde{R}}_k(1,1) & \mathbf{\tilde{R}}_k(1,2) & \cdots & \mathbf{\tilde{R}}_k(1,r-1) & \mathbf{\tilde{R}}_k(1,r) & \mathbf{\tilde{R}}_k(1,1) & \cdots & \mathbf{\tilde{R}}_k(1,n-r) \\
    \mathbf{v}_{k,1}(2) & \mathbf{\tilde{R}}_k(2,2) & \cdots & \mathbf{\tilde{R}}_k(2,r-1) & \mathbf{\tilde{R}}_k(2,r) & \mathbf{\tilde{R}}_k(2,1) & \cdots & \mathbf{\tilde{R}}_k(2,n-r) \\
    \mathbf{v}_{k,1}(3) & \mathbf{v}_{k,2}(2) & \cdots & \mathbf{\tilde{R}}_k(3,r-1) & \mathbf{\tilde{R}}_k(3,r) & \mathbf{\tilde{R}}_k(3,1) & \cdots & \mathbf{\tilde{R}}_k(3,n-r) \\
    \vdots & \vdots & \ddots & \vdots & \vdots & \vdots & \ddots & \vdots \\
    \mathbf{v}_{k,1}(r) & \mathbf{v}_{k,2}(r-1) & \cdots & \mathbf{v}_{k,r-1}(2) & \mathbf{\tilde{R}}_k(r,r) & \mathbf{\tilde{R}}_k(r,1) & \cdots & \mathbf{\tilde{R}}_k(r,n-r) \\
    \end{bmatrix}
\end{multline}
contain the entries in $\mathbf{\tilde{R}}_k$ as well as the Householder vectors $\mathbf{v}_{k,1}$, $\mathbf{v}_{k,2}$, $\ldots$, $\mathbf{v}_{k,r-1}$.
Clearly, one obtains
\begin{equation}
    \mathbf{Q}_{k} = \mathbf{H}_{k,1} \mathbf{H}_{k,2} \cdots \mathbf{H}_{k, r-1}, \quad
    \mbox{where} \quad
    \mathbf{H}_{k,i} = \mathbf{I}_r - \frac{2}{\Vert \mathbf{v}_{k,i} \Vert^2}\mathbf{\tilde{v}}_{k,i}\mathbf{\tilde{v}}_{k,i}^T, \quad
    \mathbf{\tilde{v}}_{k,i} = \begin{bmatrix}
    \mathbf{0}_{i-1} \\
    \mathbf{v}_{k,i}
    \end{bmatrix},
\end{equation}
as a result.

\section{Tangent Spaces from Laplacian Eigenmaps}
\label{app:LapEigAndTanSpace}

\subsection{The Graph Laplacian}
In this section, we follow \cite{Belkin2003}.
Let us begin by defining a graph $\mathcal{G} = (\mathcal{N}, \mathcal{E})$ containing nodes $\mathcal{N}$ at the training points and edges $\mathcal{E}$ between nodes $i$ and $j$ if $j$ is one of $i$'s $k$-nearest neighbors or vice versa.
Also, for convenience, let $\mathcal{N}_i = \left\lbrace k \ : \ (i,k)\in\mathcal{E} \right\rbrace$ be the set of $i$'s neighbors.

define a weight function
\begin{equation}
w_{i,k} = k\left(\left\Vert\mathbf{x}_i - \mathbf{x}_k\right\Vert/\sigma_{i,k}\right)
\end{equation}
using the training data $\mathbf{x}_1,$ $\ldots,$ $\mathbf{x}_m\in\mathcal{M}$, based on the Gaussian kernel $k(d) = \exp{\left(-\frac{1}{2} d^2\right)}$ and a symmetric local scale $\sigma_{i,k}$. 
The local scale is a user-specified constant $\alpha$ times the RMS distance:
\begin{equation}
\sigma_{i,k} = \alpha \sqrt{\frac{1}{2\vert \mathcal{N}_i\vert}\sum_{j\in\mathcal{N}_i} \left\Vert\mathbf{x}_i - \mathbf{x}_j \right\Vert^2 + \frac{1}{2\vert \mathcal{N}_k\vert}\sum_{l\in\mathcal{N}_k} \left\Vert\mathbf{x}_k - \mathbf{x}_l \right\Vert^2}.
\end{equation}

The graph Laplacian acts on functions $\phi:\mathcal{N}\rightarrow\mathbb{R}$ according to
\begin{equation}
\mathcal{L}_i [\phi] = \sum_{k\in\mathcal{N}_i} w_{i,k} \phi_k - \underbrace{\left(\sum_{k\in\mathcal{N}_i} w_{i,k}\right)}_{d_i} \phi_i, \quad \forall i\in\mathcal{N}.
\end{equation}
The Laplacian represents a kind locally-weighted averaging or diffusion process on the graph.
We look for the eigenfunctions $\mathcal{L}[\phi^{(j)}] = -\lambda_j \phi^{(j)}$ with smallest strictly positive eigenvalues $\lambda_j>0$ since these yield the parameterizing coordinates. 
Due to well-known results in harmonic analysis, the eigenfunction values $\phi_i^{(j)},$ $j=1,\ldots,r$ give intrinsic coordinates for $\mathbf{x}_i\in\mathcal{M}$.

\subsection{Tangent Spaces from the Graph Laplacian}
Notice that if we have an eigenfunction $\mathcal{L}[\phi^{(j)}] = -\lambda_j \phi^{(j)}$, then we can write its value at $i\in\mathcal{N}$ as
\begin{equation}
\phi_i^{(j)} = (d_i - \lambda_j)^{-1}\sum_{k\in\mathcal{N}_i} w_{i,k} \phi_k^{(j)}.
\end{equation}
We may interpolate this eigenfunction in a small neighborhood of $\mathbf{x}_i$ by letting the weights vary naturally in space according to
\begin{equation}
\tilde{w}_{i,k}(\mathbf{x}) = k\left(\left\Vert\mathbf{x} - \mathbf{x}_k\right\Vert/\sigma_{i,k}\right), \quad 
\tilde{d}_i(\mathbf{x}) = \sum_{k\in\mathcal{N}_i}\tilde{w}_{i,k}(\mathbf{x}).
\end{equation}
The interpolated eigenfunction is then given by
\begin{equation}
\tilde{\phi}_i^{(j)}(\mathbf{x}) = (\tilde{d}_i(\mathbf{x}) - \lambda_j)^{-1}\sum_{k\in\mathcal{N}_i} \tilde{w}_{i,k}(\mathbf{x}) \phi_k^{(j)}.
\end{equation}
We find a basis for the tangent space to $\mathcal{M}\subset\mathbb{R}^n$ at a training point $\mathbf{x}_i$ by computing the gradients of the interpolated eigenfunctions
\begin{equation}
\nabla \tilde{\phi}_i^{(j)}(\mathbf{x}_i) = (d_i - \lambda_j)^{-1}\left[ \sum_{k\in\mathcal{N}_i} \left(\phi_k^{(j)} - \phi_i^{(j)}\right) \nabla \tilde{w}_{i,k}(\mathbf{x}_i)\right],
\end{equation}
where
\begin{equation}
\nabla \tilde{w}_{i,k}(\mathbf{x}_i) = \frac{k'\left(\left\Vert\mathbf{x}_i - \mathbf{x}_k\right\Vert/\sigma_{i,k}\right)}{\sigma_{i,k}\left\Vert\mathbf{x}_i - \mathbf{x}_k\right\Vert}\left(\mathbf{x}_i - \mathbf{x}_k\right).
\end{equation}
If we have a collection of eigenfunctions $\phi^{1},$ $\ldots,$ $\phi^{r'}$, $r'\geq r$ paramterizing the graph nodes $\mathcal{N}$ then we can interpolate them on the manifold near $\mathbf{x}_i$ by forming \\
$\boldsymbol{\tilde{\Phi}}_i(\mathbf{x}) = \begin{bmatrix} \tilde{\phi}_i^{(1)}(\mathbf{x}) & \cdots & \tilde{\phi}_i^{(r')}(\mathbf{x}) \end{bmatrix}^T$.
Then, an orthonormal basis $\mathbf{U}_i$ for the tangent space at $\mathbf{x}_i$ can be found by computing the SVD of the eigenfunction gradients
\begin{equation}
    \nabla \boldsymbol{\tilde{\Phi}}_i(\mathbf{x}_i)^T = \mathbf{U}_i \boldsymbol{\Sigma}_i \mathbf{V}_i^T
\end{equation}
and truncating to rank-$r$.

\begin{remark}
    Since the derivatives of the Laplacian eigenfunctions become small perpendicular to the boundaries of the manifold, it is sometimes necessary to neglect these points when finding tangent planes. We simply removed the rank-deficient points when the above method was applied. More work is needed to robustly identify the boundary points and estimate their tangent planes.
\end{remark}

\section{Toy Models}
\label{app:ToyModels}
All of the following examples are contrived in order to illustrate various aspects of nonlinear DEIM using the SimPQR algorithm.

\subsection{2D Surface in 10D}
\label{subapp:2DSurf}
This example presents another case where using DEIM fails to produce a set of coordinates from which the full state can be reconstructed.
Of course DEIM is not guaranteed to fail, in fact it might even choose the optimal set of coordinates, but this would be an accident because it is not guaranteed to do so unless all of the data's variance is captured by the chosen subspace.
On the other hand NLDEIM is guaranteed, due to its construction, to provide a (possibly sub-optimal) choice of coordinates that can be used to reconstruct points on the approximate underlying manifold.
NLDEIM is possibly sub-optimal due to its greedy method of choosing coordinates.

In creating the following example, we find that we can force DEIM to choose essentially any of the coordinates simply by scaling those coordinates so that low-dimensional POD subspaces are nearly aligned with them.
On the other hand, Nonlinear DEIM (NLDEIM) using SimPQR cannot be tricked by scaling the original coordinates.

Consider a surface given by
\begin{equation}
\mathcal{M} = \left\lbrace
    \begin{bmatrix}
    \textcolor{blue}{x_1} \\
    x_2 \\
    x_3 \\
    x_4 \\
    x_5 \\
    x_6 \\
    x_7 \\
    x_8 \\
    x_9 \\
    \textcolor{blue}{x_{10}}
    \end{bmatrix} = 
    \begin{bmatrix}
    \textcolor{blue}{x_1} \\
    \sin (\pi \textcolor{blue}{x_1}) \\
    \cos (\pi \textcolor{blue}{x_1}) \\
    \sin (\pi \textcolor{blue}{x_{10}}) \\
    \cos (\pi \textcolor{blue}{x_{10}}) \\
    2.0 \sin (\pi \textcolor{blue}{x_1}) \sin (\pi \textcolor{blue}{x_{10}}) \\
    \sin (\pi \textcolor{blue}{x_1}) \cos (\pi \textcolor{blue}{x_{10}}) \\
    2.0 \cos (\pi \textcolor{blue}{x_1}) \sin (\pi \textcolor{blue}{x_{10}}) \\
    \cos (\pi \textcolor{blue}{x_1}) \cos (\pi \textcolor{blue}{x_{10}}) \\
    \textcolor{blue}{x_{10}}
    \end{bmatrix}\in\mathbb{R}^{10}: \quad
    (\textcolor{blue}{x_1}, \textcolor{blue}{x_{10}})\in [-1,1]\times[-1,1] \right\rbrace.
    \label{eqn:2DsurfMfd}
\end{equation}
The $x_1, x_{10}, x_{4}$ projection of $m=1000$ data points randomly sampled, $x_1\sim U([-1,1])$, $x_{10}\sim U([-1,1])$, from this surface is plotted in \cref{fig:2dSurfIn10D_TangentPlanes}.
We see from \cref{fig:2DSurf_SingularVals} that the data is not accurately captured by an approximation in any low-dimensional subspace.
Indeed, the manifold \cref{eqn:2DsurfMfd} is constructed so that $x_2$, $\ldots,$ $x_9$ are orthogonal functions of the uniformly sampled underlying coordinates $x_1$ and $x_{10}$.
Furthermore, we can force the leading POD subspace to align more closely with a chosen set of coordinates simply by multiplying them by larger factors;
in this example, we have scaled $x_6$ and $x_8$ by a factor of $2$.
Performing DEIM using the leading principal components of the sampled data selects coordinates $x_6$ and $x_8$ since the leading $2$-dimensional POD subspace is only $12.3^{\circ}$ away from the $x_6, x_8$ coordinate plane.
Projecting the data into the DEIM coordinate plane and coloring the points according to the leading graph Laplacian eigenfunctions in \cref{fig:2DSurf_LaplaceEfuns_DEIMcoords} reveals that the manifold is not embedded in these coordinates and hence the state cannot be reconstructed from $x_6$ and $x_8$. 
\begin{figure}
    \centering
    \includegraphics[width=0.5\linewidth]{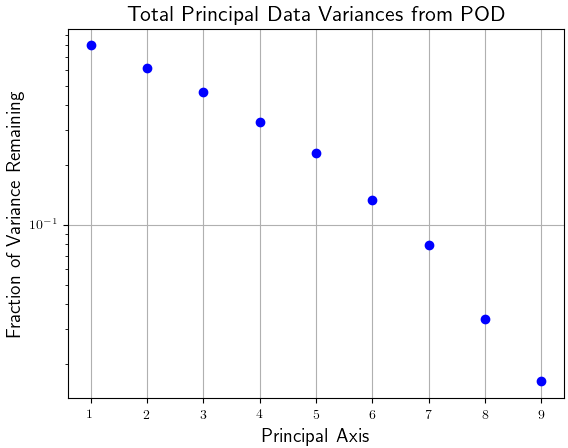}
    \caption{Fractions of total variance captured by the leading principal components of the 2D surface data set. We see that these data do not lie in a low-dimensional subspace even though the manifold is intrinsically 2D.}
    \label{fig:2DSurf_SingularVals}
\end{figure}

\begin{figure}
  \centering
  \subfloat[]{
    \includegraphics[width=0.45\textwidth]{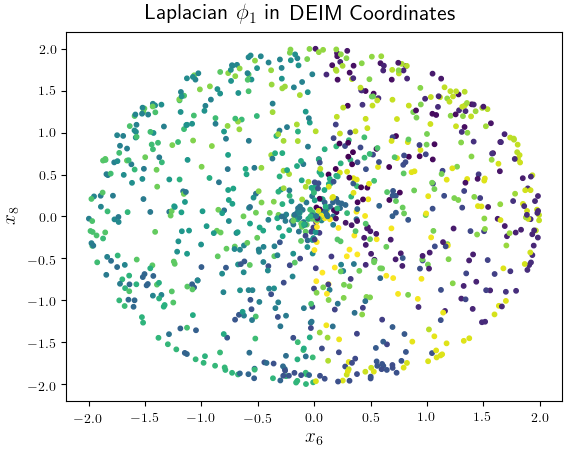}
    \label{fig:2DSurf_Phi1_DEIMcoords}
  }
  \hfill
  \subfloat[]{
    \includegraphics[width=0.45\textwidth]{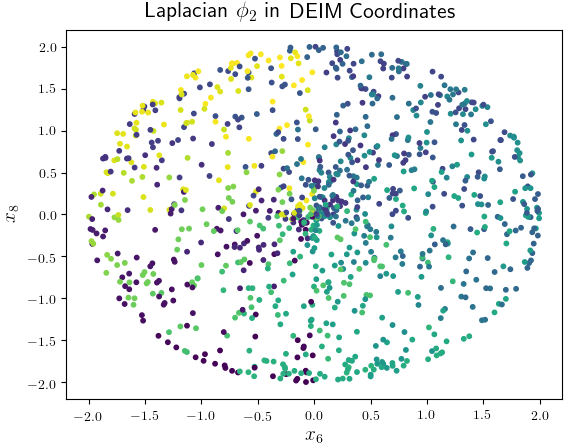}
    \label{fig:2DSurf_Phi2_DEIMcoords}
  }
  \label{fig:2DSurf_LaplaceEfuns_DEIMcoords}
  \caption{Example data points lying on the manifold \cref{eqn:2DsurfMfd} projected into the coordinate system identified using DEIM and colored by the leading graph Laplacian eigenfunction values.
  Since the Laplacian eigenfunction values are not single-valued functions of the chosen coordinates, the full state cannot be reconstructed as a function of $x_6$ and $x_8$.}
\end{figure}

On the other hand, we use the method presented in \cref{app:LapEigAndTanSpace} to compute tangent planes at each training data point and apply SimPQR with $\gamma = 1/K = 10^{-3}$ and $\epsilon = 10^{-6}$ to reveal the immersing NLDEIM coordinates $x_1$ and $x_{10}$.
These coordinates are sufficient to embed the underlying manifold \cref{eqn:2DsurfMfd} as illustrated by its definition as well as in \cref{fig:2DSurf_LaplaceEfuns_NLDEIMcoords} where the leading graph Laplacian eigenfunction are plotted in the NLDEIM coordinate system.
\begin{figure}
  \centering
  \subfloat[]{
    \includegraphics[width=0.45\textwidth]{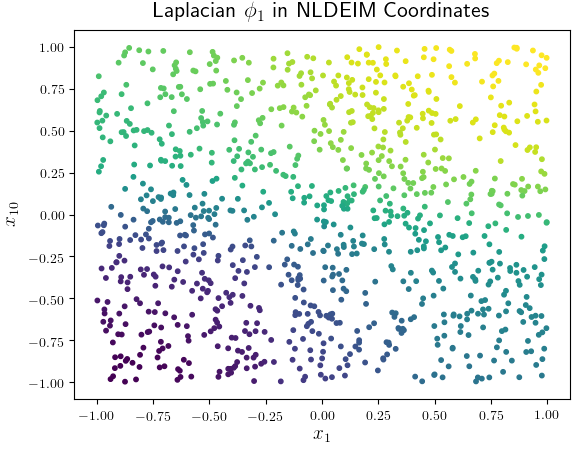}
    \label{fig:2DSurf_Phi1_NLDEIMcoords}
  }
  \hfill
  \subfloat[]{
    \includegraphics[width=0.45\textwidth]{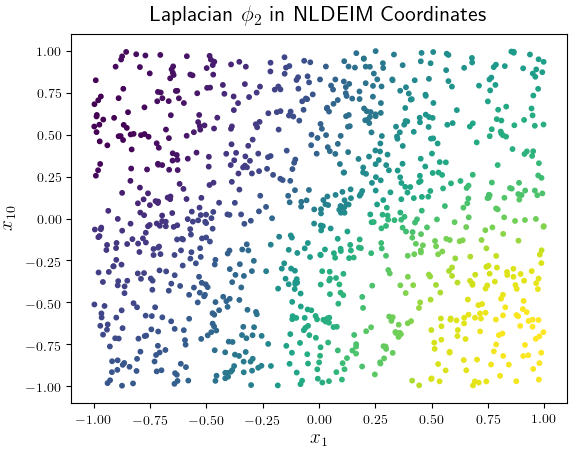}
    \label{fig:2DSurf_Phi2_NLDEIMcoords}
  }
  \label{fig:2DSurf_LaplaceEfuns_NLDEIMcoords}
  \caption{Example data points lying on the manifold \cref{eqn:2DsurfMfd} projected into the coordinate system identified using NLDEIM and colored by the leading graph Laplacian eigenfunction values.
  Since the Laplacian eigenfunction values are single-valued functions of the chosen coordinates, the full state can be reconstructed as a function of $x_1$ and $x_{10}$.}
\end{figure}

Next it is possible to investigate the performance of DEIM and NLDEIM in terms of maximum patch-wise amplification on this example when more coordinates are selected by each method.
The leading $d$ principal components of the training data were used to find $d$ DEIM coordinates for $d=2,3,\ldots,9$.
The method described in \cref{subsec:GammaPath} was used to sweep over every value of the parameter $\gamma$ which ranged from $\gamma = 10^{-3}$ up to a maximum value of $\gamma = 6.30$.
For each unique number of coordinates $d$, the choice that produced the smallest maximum amplification over the patches was recorded.
The maximum and mean disturbance amplification levels over the tangent planes at the training data points are plotted in \cref{fig:2DSurf_Amplification}.
We see that small numbers of NLDEIM coordinates offer more robust reconstruction of the full state than small numbers of DEIM coordinates on this example.
\begin{figure}
    \centering
    \includegraphics[width=0.5\linewidth]{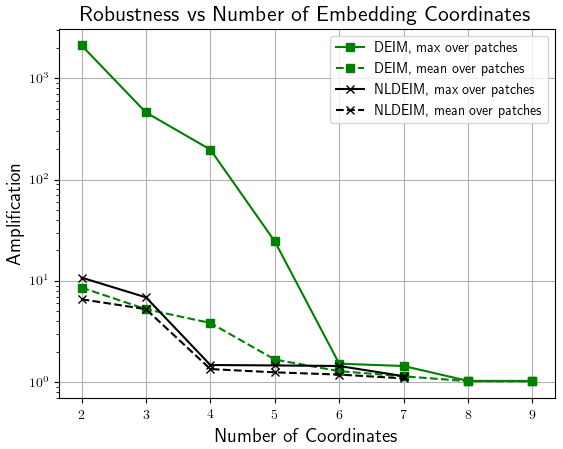}
    \caption{Maximum and mean amplification, $\sigma_{max}\left[ (\boldsymbol{\Pi}_{\mathscr{P}}^T\mathbf{U}_{\mathbf{\bar{x}}})^{+} \right]$, over tangent planes at each training data point on the 2D surface in 10D described by \cref{eqn:2DsurfMfd} using coordinates selected by DEIM and NLDEIM.
    We see that small numbers of DEIM coordinates provide much less robust reconstructions of the full state than NLDEIM coordinates on this example.}
    \label{fig:2DSurf_Amplification}
\end{figure}

\subsection{2D Cylindrical Manifold}
\label{subapp:cylMfd}
In this example, we consider a surface given by
\begin{equation}
\mathcal{M} = \left\lbrace
    \begin{bmatrix}
    \textcolor{blue}{x_1} \\
    \textcolor{blue}{x_2} \\
    \textcolor{blue}{x_3}
    \end{bmatrix} = 
    \begin{bmatrix}
    \cos{\theta} \\
    \sin{\theta} \\
    \textcolor{blue}{x_3}
    \end{bmatrix}\in\mathbb{R}^{3}: \quad
    (\theta, \textcolor{blue}{x_3})\in [0,2\pi]\times[-1,1] \right\rbrace.
    \label{eqn:CylMfd}
\end{equation}
The surface is intrinsically $2$-dimensional, but cannot be parameterized by any two of the original coordinates as illustrated in \cref{fig:CylMfd}.
Nonlinear DEIM using SimPQR successfully recovers the parameterizing coordinate indices $\mathscr{P} = \lbrace 1, 2, 3 \rbrace$ for every $\gamma \in (0, \infty)$ as long as the value of $\epsilon$ is not too small.
The parameter $\epsilon$ determines how large the columns $\mathbf{c}_k(j)$ in SimPQR \cref{alg:SimPQR} must be in order to make them available for pivoting on a given patch.
If $\mathbf{c}_k(j) < \epsilon$, then the coordinate $j$ has already been eliminated on patch $k$, in that case $\mathbf{c}_k(j) = 0$, or the tangent plane is nearly perpendicular to the $j$-axis.
Since we are using a finite number of randomly chosen patches on the manifold \cref{eqn:CylMfd}, there will not be any patch that is exactly perpendicular to $x_1$ or $x_2$.
The value of $\epsilon$ determines how close to perpendicular a patch must be to a given coordinate in order for that coordinate to be unusable for pivoting.
In this example, taking $\epsilon$ as small as $10^{-5}$ still reveals that there are tangent planes perpendicular to $x_1$ and $x_2$.
However, taking $\epsilon = 10^{-6}$ and $\gamma = 1/K = 10^{-3}$, SimPQR reveals that all of the tangent planes we randomly chose in this example can be projected into $x_1$ and $x_3$, which of course they can be, but the result is nearly singular because we have to use columns with very small magnitudes for pivoting.
If employ SimPQR with $\epsilon = 10^{-6}$ and choose the next larger value of $\gamma$ that produces a change in the solution, then we again recover the correct coordinates since this choice gives slightly more weight to patch volume.

\subsection{Spiral Manifold}
\label{subapp:SpiralMfd}
In this section we consider a spiral-shaped manifold
\begin{equation}
\mathcal{M} = \left\lbrace
    \begin{bmatrix}
    \textcolor{blue}{x_1} \\
    \textcolor{blue}{x_2} \\
    \textcolor{blue}{x_3}
    \end{bmatrix} = 
    \begin{bmatrix}
    r \cos{\theta} \\
    r \sin{\theta} \\
    \frac{\theta}{2\pi}
    \end{bmatrix}\in\mathbb{R}^{3}: \quad
    (r, \theta)\in [0.2,1]\times [-\pi,\pi] \right\rbrace.
    \label{eqn:SpiralMfd}
\end{equation}
The surface is intrinsically $2$-dimensional, but like the cylinder, cannot be parameterized by any two of the original coordinates.
\Cref{fig:SpiralMfd_ImmersionDiagram_Planes} shows of $K=m=1000$ points sampled uniformly $r\sim U([0.2,1])$, $\theta\sim U([-\pi,\pi])$ on the manifold and tangent planes computed using the method in \cref{app:LapEigAndTanSpace}.
However, unlike the cylinder in \cref{eqn:CylMfd}, SimPQR using $\epsilon = 10^{-4}$ and any value of $\gamma\in(0,\infty)$ identifies only two coordinates, $x_1$ and $x_2$, that immerse, but do not embed the manifold \cref{eqn:SpiralMfd}.

We use the simple local PCA-based technique of section \cref{sec:BranchingCoordinates} to discover the branching coordinate $x_3$.
Neighbors to each training data point were found by projecting into the $x_1,x_2$ coordinates plane and finding all points within a radius $\rho = 0.1$.
An example of such a neighborhood is shown in \cref{fig:CoordinateNhd_2D}.
On each collection of neighboring points in $x_1,x_2$ coordinates, local PCA was performed using all coordinates $x_1,x_2,x_3$ and any principal component whose singular value was greater than $2\rho$ was retained.
In this case, these principal components are vectors pointing between the clusters of neighboring points in $x_1,x_2$ coordinates on different levels or branches of the spiral manifold shown in \cref{fig:CoordinateNhd_3D}.
Performing SimPQR with $\epsilon=10^{-4}$ and $\gamma = 1/K = 10^{-3}$ on these vectors yields the branching coordinate $x_3$ that is needed to resolve the discrete ambiguities between points on \cref{eqn:SpiralMfd} projected into the $x_1,x_2$ plane.
Improved techniques for finding branching coordinates could be a subject for future work.

\section{Proofs}
\label{app:Proofs}

\begin{proof}[proof of \cref{prop:volumeSurrogate}]
By one of the interlacing properties of singular values \cite{Thompson1972principal, Queiro1987interlacing}, the $l$th singular value of a sub-matrix is always less than or equal to the $l$th singular value of the full matrix.
Since $\mathbf{U}\in\mathbb{R}^{n\times r}$ is orthonormal, its first $r$ singular values are all equal to $\sigma_1(\mathbf{U}) = \cdots = \sigma_r(\mathbf{U}) = 1$.
Since $\boldsymbol{\Pi}_{\mathscr{P}}^T\mathbf{U}$ is a sub-matrix of $\mathbf{U}$ it follows that $\sigma_l(\boldsymbol{\Pi}_{\mathscr{P}}^T\mathbf{U}) \leq 1$ for all $l=1,\ldots,r$.
The first desired inequality holds because
\begin{equation}
    \prod_{l=1}^{r-1} \sigma_{l}\left(\boldsymbol{\Pi}_{\mathscr{P}}^T\mathbf{U}\right) \leq 1 \quad \Rightarrow \quad
    \sigma_{r}\left(\boldsymbol{\Pi}_{\mathscr{P}}^T\mathbf{U}\right) \geq 
    \prod_{l=1}^r \sigma_{l}\left(\boldsymbol{\Pi}_{\mathscr{P}}^T\mathbf{U}\right).
\end{equation}

The volume, $\Vol\left(\boldsymbol{\Pi}_{\mathscr{P}}^T\mathbf{U}\right) = \left\vert \det{\left( [\boldsymbol{\Pi}_{\mathscr{P}}^T\mathbf{U}]_{\mathscr{I},:} \right)} \right\vert$ is given by the absolute value of the determinant of a square sub-matrix $[\boldsymbol{\Pi}_{\mathscr{P}}^T\mathbf{U}]_{\mathscr{I},:}$ of $\boldsymbol{\Pi}_{\mathscr{P}}^T\mathbf{U}$.
Since the absolute value of the determinant of any square matrix is the product of its singular values, the second desired inequality follows from the interlacing property:
\begin{equation}
    \prod_{l=1}^{r} \sigma_r\left( \boldsymbol{\Pi}_{\mathscr{P}}^T\mathbf{U} \right) \geq 
    \prod_{l=1}^{r} \sigma_r\left( [\boldsymbol{\Pi}_{\mathscr{P}}^T\mathbf{U}]_{\mathscr{I},:} \right) = 
    \Vol\left(\boldsymbol{\Pi}_{\mathscr{P}}^T\mathbf{U}\right).
\end{equation}
\end{proof}

\begin{proof}[proof of \cref{thm:gamma_bound}]
Choose any $1\leq j\leq n$ and $\mathscr{S}\subseteq\mathscr{S}_j$ satisfying $\sqrt{\mathbf{c}_k(j)} < \eta \sqrt{C_k}$ for some $k\in\mathscr{S}$.
Then the objective in \cref{eqn:SimPQR_optimalPivoting} is bounded above by
\begin{equation}
    \frac{\vert \mathscr{S} \vert}{S_{max}} + 
    \gamma\cdot \min_{k\in\mathscr{S}} \sqrt{\frac{\mathbf{c}_k(j)}{C_k}}
    < 1 + \gamma\eta.
\end{equation}
Since we can always find $\tilde{j}$ and at least one $\tilde{k}\in\mathscr{S}_{\tilde{j}}$ so that $\mathbf{c}_{\tilde{k}}(\tilde{j}) = C_k$, the value of the objective at $j^*$ and $\mathscr{S}^*$ can be bounded below by the objective at $\tilde{j}$ and $\mathscr{S}=\lbrace \tilde{k} \rbrace$, giving
\begin{equation}
    \frac{\vert \mathscr{S}^* \vert}{S_{max}} + 
    \gamma\cdot \min_{k\in\mathscr{S}^*} \sqrt{\frac{\mathbf{c}_k(j^*)}{C_k}}
    \geq \frac{1}{S_{max}} + \gamma \geq \frac{1}{K} + \gamma.
\end{equation}
Choosing $\gamma \geq \frac{K-1}{K(1-\eta)}$ ensures that $\frac{1}{K} + \gamma \geq 1 + \gamma\eta$, hence $\sqrt{\mathbf{c}_k(j^*)} < \eta \sqrt{C_k}$ for any $k\in\mathscr{S}^*$ produces a contradiction.
Therefore, at the optimal solution $j^*$, $\mathscr{S}^*$, we must have $\sqrt{\mathbf{c}_k(j^*)} \geq \eta \sqrt{C_k}$ for all $k\in\mathscr{S}^*$.
\end{proof}

\begin{proof}[proof of \cref{thm:gamma_simul}]
To prove the first statement, assume that $\gamma \leq 1-\nu$. 
When $S_{max}=1$ the result is trivial.
For any choice of coordinate $j$ and patches $\mathscr{S}$ where fewer than $\vert \mathscr{S} \vert \leq \nu S_{max}$ patches are chosen, the objective is upper bounded by
\begin{equation}
    \frac{\vert \mathscr{S} \vert}{S_{max}} + 
    \gamma\cdot \min_{k\in\mathscr{S}} \sqrt{\frac{\mathbf{c}_k(j)}{C_k}} \leq
    \nu + \gamma \leq 1.
\end{equation}
On the other hand, we can consider the coordinate $\tilde{j}$ and collection of patches $\tilde{\mathscr{S}}$ where $\vert \tilde{\mathscr{S}} \vert = S_{max}$.
This provides a lower bound
\begin{equation}
    \frac{\vert \mathscr{S}^* \vert}{S_{max}} + 
    \gamma\cdot \min_{k\in\mathscr{S}^*} \sqrt{\frac{\mathbf{c}_k(j^*)}{C_k}} > 1,
\end{equation}
on the optimal objective.
Therefore, the number of simultaneous patch indices at the optimal choice $j^*$ and $S^*$ must be lower bounded by $\vert \mathscr{S}^* \vert > \nu S_{max}$, otherwise there is a contradiction.

To prove the second statement, let $\gamma \leq 1/K$.
For any choice $j$ and $\mathscr{S}$ with $\vert \mathscr{S} \vert < S_{max}$, the objective is upper bounded by
\begin{equation}
     \frac{\vert \mathscr{S} \vert}{S_{max}} + 
    \gamma\cdot \min_{k\in\mathscr{S}} \sqrt{\frac{\mathbf{c}_k(j)}{C_k}} 
    \leq \frac{S_{max}-1}{S_{max}} + \gamma
    \leq 1 - \frac{1}{K} + \gamma \leq 1.
\end{equation}
Since the objective at the optimal choice $j^*$ and $\mathscr{S}^*$ is lower bounded by the value $1$, it follows that we must have $\vert \mathscr{S}^* \vert = S_{max}$.
\end{proof}

\begin{proof}[proof of \cref{lem:maximumVolume}]
There are finitely many choices of patch $k$ and coordinate $j$, hence we can find $\tilde{j}$ and $\tilde{k}$ attaining the second largest normalized magnitude $\tilde{\Theta} = \sqrt{\mathbf{c}_{\tilde{k}}(\tilde{j}) / C_{\tilde{k}}}$.
In other words $\tilde{\Theta} < 1$, but there is no $j$ and $k$ with $\tilde{\Theta} < \sqrt{\mathbf{c}_{k}(j) / C_{k}} < 1$.
Choose
\begin{equation}
    \gamma > \gamma_N = \frac{K-1}{K (1-\tilde{\Theta})}.
\end{equation}
There exists a coordinate index $j$ and a non-empty list of patches $\mathscr{S}$ where $\sqrt{\mathbf{c}_k(j) / C_k} = 1$ for all $k\in\mathscr{S}$, thus the optimal objective is bounded below by
\begin{equation}
    \frac{\vert \mathscr{S}^* \vert}{S_{max}} + 
    \gamma\cdot \min_{k\in\mathscr{S}^*} \sqrt{\frac{\mathbf{c}_k(j^*)}{C_k}} \geq \frac{1}{S_{max}} + \gamma.
\end{equation}
On the other hand, for any coordinate index $j$ and patch index list $\mathscr{S}$ where $\exists k\in\mathscr{S}$ having $\sqrt{\mathbf{c}_k(j) / C_k} < 1$, then the objective is bounded above by
\begin{equation}
    \frac{\vert \mathscr{S} \vert}{S_{max}} + 
    \gamma\cdot \min_{k\in\mathscr{S}} \sqrt{\frac{\mathbf{c}_k(j)}{C_k}} \leq 1 + \gamma\cdot \tilde{\Theta}.
\end{equation}
Choosing $\gamma > \gamma_N$ ensures that
\begin{equation}
    \frac{1}{S_{max}} + \gamma \geq \frac{1}{K} + \gamma > 1 + \gamma\cdot \tilde{\Theta},
\end{equation}
producing a contradiction unless $\mathbf{c}_k(j^*) = C_k$ for all $k\in\mathscr{S}^*$.
Therefore, optimizing the objective in \cref{eqn:SimPQR_optimalPivoting} with $\gamma > \gamma_N$ ensures that $j^*$ is a Businger-Golub pivot column for all the patches in $\mathscr{S}^*$.
\end{proof}

\begin{proof}[proof of \cref{thm:PQR_equivalence}]
Businger-Golub pivoting applied to patch $k$ at the $i$th stage yields a colum choice $j$ satisfying $\mathbf{c}_k(j) = C_k$.
Since the choice may not be unique, there may be many possible PQR factorizations.
We prove that for sufficiently large $\gamma$, SimPQR will realize one of these possible PQR factorizations on each patch.

SimPQR terminates after a finite number of stages since at least a single column and patch index are always chosen at each stage and there are a finite number of patches with finite rank.
Hence, there can be at most $r K$ stages.

At the $i$th stage, \cref{lem:maximumVolume} provides $\gamma_i < \infty$ so that choosing $\gamma > \gamma_i$ ensures that SimPQR chooses a Businger-Golub pivot column on each selected patch.
Taking $\gamma > \gamma_1$ ensures that a Businger-Golub pivot column is selected for each chosen patch during the first stage.
Assume that taking $\gamma > \max \lbrace \gamma_1, \ldots, \gamma_{i-1} \rbrace$ ensures that a Businger-Golub pivot column is selected for each chosen patch during stages $1, \ldots, i-1$.
Hence, on a given patch at stage $i$, a sequence of Businger-Golub pivot columns have already been chosen to perform Householder reflections.
If $\gamma > \max \lbrace \gamma_1, \ldots, \gamma_{i} \rbrace$ then, by the assumption, a sequence of Businger-Golub pivot columns have already been used to perform Householder reflections on each patch.
\Cref{lem:maximumVolume} guarantees that a Businger-Golub pivot column will also be chosen at the $i$th stage on all selected patches.
Therefore, if the algorithm terminates after $N$ stages, then choosing $\gamma > \max \lbrace \gamma_1, \ldots, \gamma_N \rbrace$ forces SimPQR to perform a Businger-Golub pivoted PQR factorization on each patch by induction.
\end{proof}

\begin{proof}[proof of \cref{prop:finiteBranches}]
Since the manifold is smooth and $\boldsymbol{\Pi}_{\mathscr{P}_I}^T$ is an immersion of $\mathcal{M}$ into $\mathbb{R}^d$, there exists an open neighborhood $\mathcal{U}_{\mathbf{x}}\subseteq\mathcal{M}$ around each $\mathbf{x}\in\mathcal{M}$ so that $\boldsymbol{\Pi}_{\mathscr{P}_I}^T(\mathbf{y} - \mathbf{z}) = \mathbf{0}$ if and only if $\mathbf{y} = \mathbf{z}$ for all $\mathbf{y},\mathbf{z}\in\mathcal{U}_{\mathbf{x}}$.
Since $\mathcal{M}$ is compact and $\lbrace \mathcal{U}_{\mathbf{x}} \rbrace_{\mathbf{x}\in\mathcal{M}}$ covers $\mathcal{M}$, there is a finite collection $\lbrace \mathcal{U}_1,$ $\ldots,$ $\mathcal{U}_N\rbrace$ that also covers $\mathcal{M}$.
Observe that any two distinct points $\mathbf{y},\mathbf{z}\in[\mathbf{x}]_{\mathscr{P}_I}$ must belong to two different sets $\mathbf{y}\in\mathcal{U}_i$ and $\mathbf{z}\in\mathcal{U}_j$ with $i\neq j$, for if they belong to the same set $\mathbf{y},\mathbf{z}\in\mathcal{U}_i$, then $\boldsymbol{\Pi}_{\mathscr{P}_I}^T(\mathbf{y} - \mathbf{z}) = \mathbf{0}$ implies that $\mathbf{y}=\mathbf{z}$ must be the same point. 
Thus, by the pigeonhole principle there can be at most $N$ distinct points in any $[\mathbf{x}]_{\mathscr{P}_I}$.
\end{proof}

\begin{proof}[proof of \cref{prop:branchRecoveryRobustness}]
If the bound on the disturbance holds, then by the triangle inequality
\begin{equation*}
    2\Vert \mathbf{n} \Vert < 
    \Vert \boldsymbol{\Pi}_{\mathscr{P}_B}^T\mathbf{y} - \boldsymbol{\Pi}_{\mathscr{P}_B}^T\mathbf{z} \Vert \leq
    \Vert \boldsymbol{\Pi}_{\mathscr{P}_B}^T(\mathbf{y}+\mathbf{n}) - \boldsymbol{\Pi}_{\mathscr{P}_B}^T\mathbf{z} \Vert + \Vert \mathbf{n}\Vert.
\end{equation*}
Subtracting $\Vert \mathbf{n}\Vert$, we obtain the result
\begin{equation*}
    \Vert \boldsymbol{\Pi}_{\mathscr{P}_B}^T(\mathbf{y} + \mathbf{n}) - \boldsymbol{\Pi}_{\mathscr{P}_B}^T\mathbf{y} \Vert \leq \Vert \mathbf{n} \Vert < \Vert \boldsymbol{\Pi}_{\mathscr{P}_B}^T(\mathbf{y}+\mathbf{n}) - \boldsymbol{\Pi}_{\mathscr{P}_B}^T\mathbf{z} \Vert.
\end{equation*}
To show that the bound is tight, let $\mathbf{n} = \frac{1}{2}\boldsymbol{\Pi}_{\mathscr{P}_B}\boldsymbol{\Pi}_{\mathscr{P}_B}^T(\mathbf{z} - \mathbf{y})$.
Then clearly, $\Vert \mathbf{n}\Vert = \frac{1}{2} \Vert \boldsymbol{\Pi}_{\mathscr{P}_B}^T(\mathbf{y} - \mathbf{z}) \Vert$ and direct substitution shows that
\begin{equation*}
    \Vert \boldsymbol{\Pi}_{\mathscr{P}_B}^T(\mathbf{y} + \mathbf{n}) - \boldsymbol{\Pi}_{\mathscr{P}_B}^T\mathbf{y} \Vert 
    = \Vert \boldsymbol{\Pi}_{\mathscr{P}_B}^T(\mathbf{y} + \mathbf{n}) - \boldsymbol{\Pi}_{\mathscr{P}_B}^T\mathbf{z} \Vert,
\end{equation*}
hence the disturbed point is equidistant from $\mathbf{y}$ and $\mathbf{z}$ in the branching coordinates.
\end{proof}

\begin{proof}[proof of \cref{prop:StrongerBranchRecovery}]
Choose a nonzero $\mathbf{w}\in\Delta[\mathbf{x}]_{\mathscr{P}_I}$ then $\mathbf{w} = \mathbf{\tilde{U}}_{\mathbf{x}}\mathbf{a}$ for some $\mathbf{a}\neq\mathbf{0}$.
$\Vol{\left(\boldsymbol{\Pi}_{\mathscr{P}_B}^T\mathbf{\tilde{U}}_{\mathbf{x}}\right)} > 0$ means that $\boldsymbol{\Pi}_{\mathscr{P}_B}^T\mathbf{\tilde{U}}_{\mathbf{x}}$ is left-invertible and hence we can reconstruct
\begin{equation*}
    \mathbf{w} = \mathbf{\tilde{U}}_{\mathbf{x}} (\boldsymbol{\Pi}_{\mathscr{P}_B}^T\mathbf{\tilde{U}}_{\mathbf{x}})^+\boldsymbol{\Pi}_{\mathscr{P}_B}^T\mathbf{w}.
\end{equation*}
It follows that
\begin{equation*}
    \Vert \mathbf{w}\Vert \leq 
    \Vert (\boldsymbol{\Pi}_{\mathscr{P}_B}^T\mathbf{\tilde{U}}_{\mathbf{x}})^+ \Vert \Vert \boldsymbol{\Pi}_{\mathscr{P}_B}^T\mathbf{w}\Vert
    \leq \frac{\Vert \boldsymbol{\Pi}_{\mathscr{P}_B}^T\mathbf{w}\Vert}{\Vol{\left(\boldsymbol{\Pi}_{\mathscr{P}_B}^T\mathbf{\tilde{U}}_{\mathbf{x}}\right)}},
\end{equation*}
yielding the desired inequality.
The bound on the disturbance is trivial.
\end{proof}

\end{document}